\documentclass[11pt,reqno]{amsart}
\usepackage{amssymb}
\usepackage[all]{xy}
\usepackage{fancyhdr}
\usepackage{hyperref}
\usepackage{tikz-cd}
\usepackage{cite}
\usepackage{amsmath,amsfonts}
\usepackage{graphicx}
\usepackage{wrapfig}
\usepackage{array}
\usepackage{amsthm}
\usepackage{etoolbox}
\patchcmd{\section}{\scshape}{\bfseries}{}{}
\makeatletter
\renewcommand{\@secnumfont}{\bfseries}
\makeatother
\usepackage[left=1.2in, right=1.2in, top=1in, bottom=1in]{geometry}

\DeclareMathOperator{\Hom}{Hom}
\DeclareMathOperator{\Spec}{Spec}
\DeclareMathOperator{\Frac}{Frac}
\DeclareMathOperator{\Img}{Img}

\input xy
\xyoption{all}
\usepackage{secdot}  

\newtheorem{mydef}{\textbf{Definition}}[section]
\newtheorem{rmk}[mydef]{\textbf{Remark}}
\newtheorem{obs}[mydef]{\textbf{Observation}}
\newtheorem{myeg}[mydef]{\textbf{Example}}
\theoremstyle{plain}
\newtheorem{mythm}[mydef]{\textbf{Theorem}}
\newtheorem{lem}[mydef]{\textbf{Lemma}}
\newtheorem{pro}[mydef]{\textbf{Proposition}}

\newtheorem{cor}[mydef]{\textbf{Corollary}}

\newtheorem*{thma}{\textbf{Theorem A}}
\newtheorem*{thmb}{\textbf{Theorem B}}
\newtheorem*{thmc}{\textbf{Theorem C}}
\newtheorem*{thmd}{\textbf{Theorem D}}
\patchcmd{\abstract}{\scshape\abstractname}{\normalsize{\textbf{\abstractname}}}{}{}
\begin{document}

\title{Valuations of semirings}
\author{Jaiung Jun}
\address{Department of Mathematical Sciences, Binghamton University, Binghamton, NY, 13902, USA}
\email{jjun@math.binghamton.edu}

\subjclass[2010]{14T99(primary), 13A18(secondary).}
\keywords{characteristic one, abstract curve, valuations of semirings, projective line.}
\thanks{}

\begin{abstract}
\normalsize{We develop notions of valuations on a semiring, with a view toward extending the classical theory of abstract nonsingular curves and discrete valuation rings to this general algebraic setting; the novelty of our approach lies in the implementation of hyperrings to yield a new definition (\emph{hyperfield valuation}). In particular, we classify valuations on the semifield $\mathbb{Q}_{max}$ (the max-plus semifield of rational numbers) and also valuations on the `function field' $\mathbb{Q}_{max}(T)$ (the semifield of rational functions over $\mathbb{Q}_{max}$) which are trivial on $\mathbb{Q}_{max}$. We construct and study the abstract curve associated to $\mathbb{Q}_{max}(T)$ in relation to the projective line $\mathbb{P}^1_{\mathbb{F}_1}$ over the field with one element $\mathbb{F}_{1}$ and the tropical projective line. Finally, we discuss possible connections to tropical curves and Berkovich's theory of analytic spaces. } 
\end{abstract}

\maketitle

\vspace*{6pt} % for this guide only.
% A table of contents should normally not be included

\section{Introduction}
We investigate notions of valuations on semirings aiming at finding an analogue of abstract nonsingular curves over semirings. First, we recall the classical definition. Let $k$ be an algebraically closed field and $K$ be a function field of dimension $1$ over $k$, i.e., $K$ is a finitely generated field extension of $k$ of transcendence degree $1$. The abstract nonsingular curve associated to $K$ over $k$ is the set, denoted by $C_K$, of all discrete valuation rings of $K$ which are trivial on $k$ equipped with a suitable topology and a sheaf of regular functions; the discrete valuation rings form the local rings at all points of the curve. The abstract nonsingular curve $C_K$ is isomorphic to a nonsingular projective curve over $k$ $($see \cite[Theorem $6.9$, \S $1.5$]{Har}$)$. In particular, the abstract nonsingular curve $C_K$ is isomorphic to the projective line $\mathbb{P}^1_k$ if and only if $K$ is a purely transcendental extension of $k$. \\ 

A semiring generalizes a ring in such a way that one does not assume the existence of additive inverses. An example of a semiring is the set $\mathbb{N}$ of natural numbers. Another example is the set $\mathbb{R}_{max}:=\mathbb{R}\cup \{-\infty\}$ of real numbers together with the symbol $-\infty$ with the addition given by maximum and the multiplication given by the usual addition of real numbers. $\mathbb{R}_{max}$ is called the \emph{max-plus algebra} or the \emph{tropical semifield} (cf. Example \ref{rmax}). Another example on which we shall focus mainly in this paper
is $\mathbb{Q}_{max}$ which is a subsemifield of $\mathbb{R}_{max}$ consisting of rational numbers together with $-\infty$;  in our investigation
of semirings, the semifield $\mathbb{Q}_{max}$ plays a role that an algebraically closed field
$k$ assumes in the classical setting (cf. \S \ref{valuationofQmax}).  \\

Recently there has been a growing interest in semiring theory among algebraic geometers. One motivation stems from tropical geometry which can be viewed as algebraic geometry over $\mathbb{R}_{max}$ and has proved to be a very effective tool in approaching enumerative problems in algebraic geometry (cf. \cite{Gri1}). For an introduction to tropical geometry, see \cite{bernd}. \\ 
Another motivation is to develop a generalized scheme theory over more general algebraic structures such as monoids or semirings. This program is commonly known as $\mathbb{F}_1$-geometry or algebraic geometry in characteristic one. An overview of $\mathbb{F}_1$-geometry can be found in \cite{pena2009mapping}. See \cite{con5} for the mathematical meaning of working in characteristic one in connection to tropical geometry. \\
Two motivations began independently; however, they are closely related. Recently, in \cite{noah}, J.Giansiracusa and N.Giansiracusa gave a generalized scheme structure (tropical scheme structure) to a tropical variety partly by means of monoid schemes which have been studied in $\mathbb{F}_1$-geometry. \\ 

Inspired by the classical case, one may hope that valuations on semirings provide some geometric information. Specifically, our paper is motivated by the question of how the classical theory of abstract nonsingular curves can be extended to semirings. For this purpose, we first present three different definitions of semiring valuations and compute basic examples before turning to abstract curves. \\

A \emph{classical valuation} (cf. Definition \ref{classcalvaluation}) directly extends the classical notion for rings. A \emph{strict valuation} (cf. Definition \ref{semiringmorphism}) reflects the fact that for a valuation $\nu$, we have $\nu(a+b) \in \{\nu(a),\nu(b)\}$ if $\nu(a) \neq \nu(b)$; this definition was first given in \cite{izhakian2011supertropical} to investigate valuation theory on a supertropical algebra which is an algebraic structure introduced by Z.Izhakian. In \cite{tolliver2016extension}, J.Tolliver recently employed strict valuations to yield a new proof for a classical extension of valuations theorem for fields. To the best of our knowledge, this notion has not been used for the purpose
we pursue in this paper.\\
Our main innovation is to come up with a \emph{hyperfield valuation} (cf. Definition \ref{semihypval}). We make use of rather unknown algebraic structures, hyperfields, for the codomains of valuations instead of the field $\mathbb{R}$ of real numbers. Roughly speaking, a hyperfield is an algebraic object which generalizes a field by allowing addition to be multi-valued (cf. \cite{baker2016matroids}, \cite{con3}, \cite{viro}, or \S \ref{hypring} for basic notions of hyperfields). A hyperfield valuation reflects a probabilistic intuition: when $\nu(a)=\nu(b)$, the value $\nu(a+b)$ is not solely determined by $\nu(a)$ and $\nu(b)$, but $\nu(a+b) \in [-\infty,\nu(a)]$. We note that a hyperfield valuation generalizes a strict valuation in a suitable sense (see Proposition \ref{generator} and $\S \ref{tropicalcurve}$).\\

We begin by applying our idea of semiring valuations to $\mathbb{Q}_{max}$. In our framework, $\mathbb{Q}_{max}$ is a tropical replacement for an algebraically closed field $k$. To be more precise, we are primarily interested in $\mathbb{Q}_{max}$ since any tropical polynomial with coefficients in $\mathbb{Q}_{max}$ has a tropical solution in $\mathbb{Q}_{max}$, i.e., $\mathbb{Q}_{max}$ is `tropically algebraically closed'. In addition, $\mathbb{Q}_{max}$ is the smallest `tropically algebraically closed' subsemifield of $\mathbb{R}_{max}$ that contains $\mathbb{Z}_{max}$; this is crucial, because $\mathbb{Z}$ is the value group of all discrete valuation rings. We first prove the following:

\begin{thma}
There are exactly three classical valuations and two strict and hyperfield valuations on $\mathbb{Q}_{max}$ up to equivalence.
\end{thma}

Next, we turn our attention to a tropical replacement for a function field $K$. We consider the case of a projective line, i.e., $K$ is a purely transcendental extension of $k$; $K:=\Frac(k[T])=k(T)$.\\
In the classical setting, one may obtain $k(T)$ from the polynomial ring $k[T]$ by localizing at $S=k[T]-\{0\}$, thus having a canonical injection $k[T] \hookrightarrow k(T)$. But, in tropical setting, things fall apart quickly from the following reasons:
\begin{enumerate}
\item 
$\mathbb{Q}_{max}[T]$ is not multiplicatively cancellative.
\item
Even if any $f \in \mathbb{Q}_{max}[T]$ has a tropical solution in $\mathbb{Q}_{max}$, $f$ does not have to be a product of linear polynomials in $\mathbb{Q}_{max}[T]$ (see, \S \ref{valuationofQmax}).
\end{enumerate} 
We address these issues by implementing the functional equivalence relation which is used in tropical literature. Then we define the function field $\mathbb{Q}_{max}(T)$ to be the localization of $\mathbb{Q}_{max}[T]/\sim$, where $\sim$ is the functional equivalence relation, and we prove the following:

\begin{thmb}
\begin{enumerate}
\item
There are exactly two non-trivial strict valuations on $\mathbb{Q}_{max}(T)$ (which are trivial on $\mathbb{Q}_{max}$) up to equivalence.
\item
The set of equivalence classes of non-trivial hyperfield valuations on $\mathbb{Q}_{max}(T)$, which are trivial on $\mathbb{Q}_{max}$, is equal to the set $\mathbb{Q}_{max} \cup \{\infty\}$. 
\end{enumerate}
\end{thmb}

By adopting these results, we define the tropical abstract curve $C_K$ associated to $K=\mathbb{Q}_{max}(T)$ over $k=\mathbb{Q}_{max}$, and study its properties. Surprisingly, when we consider strict valuations, $C_K$ recovers all closed points of the projective line $\mathbb{P}^1_{\mathbb{F}_1}$ over the field with one element $\mathbb{F}_1$. To be precise, we prove the following:

\begin{thmc}
The abstract nonsingular curve associated to $\mathbb{Q}_{max}(T)$ over $\mathbb{Q}_{max}$ (constructed by using strict valuations) is homeomorphic to the subspace of projective line $\mathbb{P}^1_{\mathbb{F}_1}$ over $\mathbb{F}_1$ which consists of all closed points.
\end{thmc}

On the other hand, when we consider hyperfield valuations, $C_K$ recovers all closed points of the tropical projective line $\mathbb{P}^1_{\mathbb{Q}_{max}}$ over $\mathbb{Q}_{max}$. To be precise, we prove the following:

\begin{thmd}
There is a canonical bijection between the abstract nonsingular curve $C_K$ associated to $\mathbb{Q}_{max}(T)$ over $\mathbb{Q}_{max}$ (constructed by using hyperfield valuations) and the set of closed points of the tropical projective line $\mathbb{P}^1_{\mathbb{Q}_{max}}$ over $\mathbb{Q}_{max}$. Furthermore, the semiring of global sections $\mathcal{O}_{C_K}(C_K)$ is isomorphic to $\mathbb{Q}_{max}$ and $\mathcal{O}_{C_K}(U)$, where $U:=C_K -\{\infty\}$ is isomorphic to $\mathbb{Q}_{max}[T]/\sim$, where $\sim$ is the functional equivalence relation on $\mathbb{Q}_{max}[T]$. 
\end{thmd}

The above theorem sheds some light on algebraic foundation of tropical curves as abstract curves. Furthermore, in this case, an (abstract) tropical curve is equipped with a structure sheaf in a canonical way which is parallel to the classical case for rings. What makes our approach of hyperfield valuations more compelling is that in the category of hyperrings, a hyperfield valuation becomes simply a homomorphism. In other words, a hyperfield valuation on a commutative ring $A$ is just a homomorphism from $A$ to the \emph{tropical hyperfield} $\mathbb{T}$ (see Example \ref{tropicalhyperfield}). In fact, in \cite{baker2016matroids}, Baker and Bowler observe that a similar idea can be used to recast basic definitions of Berkovich's theory of analytic spaces. We discuss this aspect as well as possible applications to tropical curves in \S \ref{tropicalcurve}. 

\subsection*{Acknowledgment}
This is a part of the author's Ph.D. thesis \cite{jaiungthesis}. The author would like to thank his academic advisor Caterina Consani for suggesting this project. The author also thanks Paul Lescot for helpful comments regarding Section \ref{valuationofQmax}. Finally, the author thanks the referee for helpful comments which greatly helped the author to improve the earlier drafts.  

\section{Review: semirings, hyperrings, and the projective line $\mathbb{P}^1_{\mathbb{F}_1}$}\label{review}
\subsection{Basic definitions of semirings}\label{semirings}
\noindent The following are the basic definitions of semirings which will be used in the paper.
\begin{mydef}
By a \emph{monoid} we mean a nonempty set $M$ equipped with a binary operation $\cdot$ which satisfies the following conditions: $(1)$ $(x\cdot y)\cdot z=x\cdot(y\cdot z)$ $\forall x,y,z \in M$ and $(2)$ $\exists !$ $1_M \in M$ such that $a\cdot 1_M=1_M\cdot a$ $\forall a \in M$. If $x\cdot y=y\cdot x$ $\forall x,y \in M$, we call $M$ \emph{commutative}. When $M$ only satisfies the first condition then $M$ is called a \emph{semigroup}.
\end{mydef}
\begin{mydef}
A \emph{semiring} $(S,+,\cdot)$ is a nonempty set $S$ endowed with addition $+$ and multiplication $\cdot$ such that
\begin{enumerate}
\item
$(S,+)$ is a commutative monoid with the neutral element $0$.
\item
$(S,\cdot)$ is a monoid with the identity $1$.
\item
The operations $+$ and $\cdot$ are compatible, i.e., $x\cdot(y+z)=x\cdot y+x\cdot z \textrm{ and } (x+y)\cdot z=x\cdot z+y\cdot z \quad \forall x,y,z \in S.$
\item
$0$ is an absorbing element, i.e., $m\cdot 0=0\cdot m=0\quad \forall m \in S.$
\item
$0 \neq 1$.
\end{enumerate}
When $(S,\cdot)$ is commutative, we call $S$ a \emph{commutative semiring}. If every nonzero element of a commutative semiring $S$ is multiplicatively invertible, then $S$ is called a \emph{semifield}.
\end{mydef}

In what follows we assume that all semirings and monoids are commutative. 

\begin{mydef}
Let $S_1$ and $S_2$ be semirings. A map $\varphi:S_1 \longrightarrow S_2$ is a \emph{homomorphism} of semirings if $\forall a,b \in S_1$,
\[\varphi(a+b)=\varphi(a)+\varphi(b), \quad \varphi(ab)=\varphi(a)\varphi(b), \quad \varphi(0)=0, \quad \varphi(1)=1.\] 
\end{mydef}
\noindent By an \emph{idempotent semiring}, we mean a semiring $S$ such that $x+x=x$ $\forall x \in S$.
\begin{myeg} 
Let $\mathbb{B}:=\{0,1\}$. The addition is defined by: $1+1=1$, $1+0=0+1=1$, and $0+0=0$. The multiplication is given by $1\cdot 1=1$, $1\cdot 0=0$, and $0\cdot 0=0$. $\mathbb{B}$ is called the \emph{Boolean semifield} and it is the initial object in the category of idempotent semirings. 
\end{myeg}
\begin{myeg}\label{rmax}
The \emph{tropical semifield} $\mathbb{R}_{max}$ has the underlying set $\mathbb{R}\cup \{-\infty\}$. The addition $\oplus$ is given by $a\oplus b=\max\{a,b\}$ with the usual order of the real numbers and the smallest element $-\infty$. The multiplication $\odot$ is defined as the usual addition of $\mathbb{R}$ as follows: $a \odot b:=a+b$, where $+$ is the usual addition of real numbers and $(-\infty)\odot a=a\odot (-\infty)=(-\infty)$ $\forall a \in \mathbb{R}_{max}$. By $\mathbb{Q}_{max}$ and $\mathbb{Z}_{max}$ we mean the subsemifields of $\mathbb{R}_{max}$ with the underlying sets $\mathbb{Q}\cup \{-\infty\}$ and $\mathbb{Z}\cup \{-\infty\}$ respectively.
\end{myeg}

\begin{mydef}
Let $S$ be a semiring.
\begin{enumerate}
\item
An \emph{ideal} $I$ is a subset of $S$ such that $(I,+)$ is a sub-monoid and $aI \subseteq I$ $\forall a \in S$. An ideal $I$ is \emph{proper} if $I$ is not equal to $S$. 
\item
A \emph{prime ideal} is a proper ideal $I$ such that if $ab \in I$ then $a \in I$ or $b \in I$. 
\item
A \emph{maximal ideal} is a proper ideal $I$ which is not contained in any other proper ideals. 
\end{enumerate}
\end{mydef}

\begin{mydef}
Let $M$ be a semiring and $S$ be a multiplicative subset of $M$, i.e., for any $x,y \in S$, we have $xy \in S$. One imposes the following equivalence relation $\sim$ on $M\times S$:
\[
(a,s) \sim (b,t) \iff \exists q\in S \textrm{ such that } qat=qsb.
\]
We denote the equivalence class of $(a,s)$ by $\frac{a}{s}$. The \emph{localization} $S^{-1}M$ of $M$ at $S$ is defined as the set $\{\frac{a}{s} \mid a \in M, s \in S\}$ of equivalence classes. Then the semiring structure of $S$ descends to $S^{-1}M$ as follows:
\[
\frac{a}{s}+\frac{a'}{s'}:=\frac{as'+a's}{ss'}, \quad \frac{a}{s}\cdot \frac{a'}{s'}:=\frac{aa'}{ss'}. 
\]
\end{mydef}
Recall that a semiring $M$ is (multiplicatively) \emph{cancellative} if for $a,b \in M, c \neq 0_M \in M$, $a\cdot c=b \cdot c$ implies that $a=b$. When $M$ is multiplicatively cancellative, the localization map $S^{-1}:M \longrightarrow S^{-1}M$ sending $m$ to $\frac{m}{1}$ is injective. We refer the readers to \cite{semibook} for an introduction to the theory of semirings.
\subsection{Basic definitions of hyperrings}\label{hyperrings}
\noindent In this section, we give a very brief introduction to the theory of hyperrings. In a word, a hyperring is an algebraic object which generalizes a ring by allowing addition to be \emph{multi-valued}.\\
A hyper-operation on a nonempty set $H$ is a map
$+ : H \times H \rightarrow \mathcal{P}(H)^{*}$, where $\mathcal{P}(H)^*$ is the set of nonempty subsets of $H$. For any subsets $A, B\subseteq H$, we define $A+B:=\bigcup_{a\in A, b\in B}(a+b)$. Also, for $x,y \in H$, if the set $x+y$ consists of a single element $z$, we write $x+y=z$ rather than $x+y=\{z\}$.

\begin{mydef}$($cf.\cite{con3}$)$\label{hypgroup}
A \emph{canonical hypergroup} $(H,+)$ is a nonempty pointed set with a hyper-operation $+$ which satisfies the following properties:
\begin{enumerate}
\item
$x+y=y+x\quad  \forall x,y \in H. \quad \text{(commutativity)}$
\item
$(x+y)+z=x+(y+z)\quad  \forall x, y, z\in H. \quad \text{(associativity)}$
\item
$\exists ! 0 \in H\textrm{ such that }0+x=x=x+0\quad \forall x\in H. \quad \text{(neutral element)}$
\item
$\forall x\in H \quad \exists  !  y(:=-x)\in H  \text{ such that }  0\in x+y. \quad \text{(unique inverse)}$
\item
$x \in y+z \Longrightarrow z \in x-y. \quad \text{(reversibility)}$
\end{enumerate}
\end{mydef}

\noindent The definition of a hypergroup is more general than that of a canonical hypergroup. For example, reversibility and commutativity are not assumed. However, in this paper, by a hypergroup we always mean a canonical hypergroup.

\begin{mydef}$($cf.\cite{con3}$)$\label{hypring}
A \emph{hyperring} $(R, + ,\cdot)$ is a nonempty set $R$ with a hyper-operation $+$ and a usual multiplication $\cdot$ such that $(R,+)$ is a canonical hypergroup, $(R,\cdot)$ is a commutative monoid, and two operations are distributive, i.e., $x\cdot (y+z)=x\cdot y +x \cdot z$ (as sets) for all $x,y,z \in R$. When $(R\setminus \{ 0 \}, \cdot)$ is an abelian group, we call $(R, +, \cdot)$ a \emph{hyperfield}.
\end{mydef}

\begin{rmk}
Strictly speaking, Definition \ref{hypring} is the definition of \emph{Krasner hyperring}. The definition of a hyperring is more general. However, since we will be only using Definition \ref{hypring}, by a hyperring we will always mean Definition \ref{hypring}. For this reason, from now on, we will simply say a \emph{hyperaddition} for a \emph{hyper-operation}.
\end{rmk}

\begin{mydef}$($cf.\cite{con3}$)$
For hyperrings $R_1$ and $R_2$, 
a map $ f : R_1 \longrightarrow R_2$ is called a \emph{homomorphism} of hyperrings if
\begin{enumerate}
\item
$f(a+b) \subseteq f(a)+f(b) \quad \forall a, b \in R_1$.
\item
$f(a\cdot b) = f(a)\cdot f(b) \quad \forall a, b \in R_1$.
\item
We call $f$ \emph{strict} if, moreover, $f(a+b)= f(a)+f(b) \quad \forall a, b \in R_1$.
\end{enumerate}
\end{mydef}

\begin{myeg}$($cf.\cite{con3}$)$
\begin{enumerate}
\item
Let $\mathbf{K}:=\{0,1\}$. The multiplication is given by 
\[1\cdot 1=1, \quad 0\cdot 1=1 \cdot 0=0,\] 
and the hyperaddition is given by 
\[0+1=\{1\},\quad  0+0=\{0\}, \quad  1+1=\{0,1\}.\]
Then $(\mathbf{K},+,\cdot)$ is a hyperfield called the \emph{Krasner hyperfield}.
\item
Let $\mathbf{S} := \{-1,0,1\}$. The multiplication is given by
\[1\cdot 1=(-1)\cdot(-1)=1,\quad  (-1)\cdot 1=(-1),\quad  a\cdot 0=0 \quad \forall a \in \mathbf{S}.\]
The hyperaddition $+$ is given by
\[0+0=0, 1+0=1+1=1, (-1)+0=(-1)+(-1)=(-1), 1+(-1)=\mathbf{S}.\]
We call $\mathbf{S}$ the \emph{hyperfield of signs}.
\end{enumerate}
\end{myeg}
\noindent One can easily obtain hyperrings from commutative rings as follows.

\begin{mythm}$($cf.\cite[Proposition $2.6$]{con3}$)$
Let $A$ be a commutative ring and $G \subseteq A^\times$ be a subgroup of the multiplicative group $A^\times$ of units. Then $G$ induces an equivalence relation $\sim$ on $A$: 
\[
x\sim y \textrm{ if and only if } x=gy \textrm{ for some } g \in G.
\]
 Then, the set $A/G$ (of equivalence classes of $\sim$) is a hyperring with the following operations:
\begin{enumerate}
\item
$xG\cdot yG := xyG \quad \forall x, y \in A$. (multiplication)
\item
$xG+yG:= \{zG \mid z=xa+yb \textrm{ for some }a,b \in G\} \quad \forall x, y \in A$. (hyperaddition)
\end{enumerate}
A hyperring which arises in this way is called a quotient hyperring.
\end{mythm}
Note that, for a field $k$ with $|k|\geq 3$, we can identify the Krasner hyperfield $\mathbf{K}$ with the quotient hyperring $k/k^\times$.\\
In \cite{con3}, Connes and Consani constructed a crucial link between a quotient hyperring and an incidence geometry. More precisely, they showed that for a hyperfield extension $\mathbb{H}$ of the Krasner hyperfield $\mathbf{K}$ one can canonically associate an (incidence) projective geometry $\mathcal{P}$. Moreover, if $\mathcal{P}$ is Desarguesian of dimension $\geq 2$, then there exists a unique pair $(F,K)$ of a field $F$ and a subfield $K \subseteq F$ such that $\mathbb{H} \simeq F/K^\times$. Furthermore, the open problem of classifying all finite hyperfield extensions of $\mathbf{K}$ is equivalent to the abelian case of the conjecture on the existence of Non-Desarguesian finite projective planes with a simply transitive group $G$ of collineations.

\begin{rmk}
In \cite{izhakian2014layered}, Z.Izhakian, M.Knebusch and L.Rowen generalize their supertropical algebras by implementing a \emph{layered structure}. The authors also provided an example (\cite[Example 11.6]{izhakian2014layered}) which shows that layered tropical algebras have certain connection to hyperfields.
\end{rmk}

\begin{rmk}
There is another algebraic structure called a \emph{fuzzy ring}, which generalizes a commutative ring, first introduced by A.Dress in \cite{dress1986duality} to unify various generalizations of matroids. In \cite{giansiracusa2016relation}, together with J.Giansiracusa and O.Lorscheid, the author constructs a faithful functor from the category of hyperrings to the category of fuzzy rings and prove that the notion of matroids over hyperfields, introduced by Baker and Bowler in \cite{baker2016matroids}, is essentially the same thing as the notion of matroids over fuzzy rings.  
\end{rmk}

\subsection{Projective line $\mathbb{P}^1_{\mathbb{F}_1}$} \label{char}
We review the definition of the projective line $\mathbb{P}^1_{\mathbb{F}_1}$ over $\mathbb{F}_1$ introduced in \cite{Deitmar} by A.Deitmar. We refer the readers to \cite{Deitmar} and \cite{deitmar2008f1} for the theory of monoid schemes. 

\begin{rmk}
Deitmar does not assume that a monoid has an absorbing element $0$, however, we use a monoid with an absorbing element. In other words, we assume that a monoid $(M,\cdot,1_M)$ has an absorbing element $0_M$ such that $m\cdot 0_M=0_M$ for all $m \in M$. This makes no difference as one can always adjoin an absorbing element to a monoid. 
\end{rmk}

An \emph{ideal} of a monoid $M$ is a subset $I$ of $M$ such that \[
IM:=\{i\cdot m \mid i \in I, \textrm{ } m \in M\} \subseteq I.
\]
In particular, this implies that an ideal $I$ contains $0_M$. A \emph{prime ideal} of $M$ is an ideal $P$ such that if $xy \in P$ then $x \in P$ or $y \in P$. Let $\Spec M$ be the set of prime ideals of $M$. One can endow Zariski topology and a structure sheaf to $\Spec M$ as in the classical case for commutative rings. We remark that one defines a scalar extension from monoids to rings as $M \otimes _{\mathbb{F}_1}\mathbb{Z}:=\mathbb{Z}[M]$, a monoidal ring, and this induces a continuous map from $\Spec \mathbb{Z}[M]$ to $\Spec M$. We refer the readers to \cite{noah} for more details on how these scalar extension functors can be used to link $\mathbb{F}_1$-geometry and tropical geometry. 
\begin{rmk}
It seems that mathematicians outside of $\mathbb{F}_1$-geometry usually consider that a monoid scheme (in \cite{Deitmar}) is the only approach to $\mathbb{F}_1$-geometry. However, the theory of monoid schemes is `a minimal approach' to $\mathbb{F}_1$-geometry, for instance see \cite{pena2009mapping}. There are other approaches which also generalize algebraic geometry over semirings. For example, see \cite{durov2007new}, \cite{oliver1}, \cite{toen2009dessous}. Connes and Consani have developed a notion of algebraic geometry in characteristic one (this is essentially algebraic geometry over idempotent algebraic structures) and linked it to $\mathbb{F}_1$-geometry. See, for example, \cite{con5}, \cite{con7}. 
\end{rmk}
\begin{myeg}
Let $M=\{1=t^0,t,t^2,...,t^n,...\} \cup \{0\}$, where $0\cdot t^n=0$ for all $n \geq 0$. We have $\Spec M =\{p_0,p_1\}$, where $p_0=\{0\}$ and $p_1=\{0,t,t^2,...\}$. One can easily see that $p_1$ is the only closed point of $\Spec M$.
\end{myeg}

Now, we review the construction of $\mathbb{P}^1_{\mathbb{F}_1}$. Let $C_{\infty}:=\{...,t^{-1},1,t,...\} \cup \{0\}$ be an infinite cyclic group generated by $t$ together with an absorbing element $0$ and let $C_{\infty,+}:=\{1,t,t^2,...\} \cup \{0\}$, $C_{\infty,-}:=\{1,t^{-1},t^{-2},...\} \cup \{0\}$ be sub-monoids of $C_{\infty}$. One can observe that 
\[U:=\Spec (C_{\infty})=\{p_0\}, \quad p_0=\{0\},\]
\[U_+:=\Spec (C_{\infty,+})=\{q_0,q_+\}, \quad q_0=\{0\},\textrm{ }q_+=\{0,t,t^2,...\},
\]
\[
U_-:=\Spec (C_{\infty,-})=\{r_0,r_-\}, \quad r_0=\{0\},\textrm{ }r_-=\{0,t^{-1},t^{-2},...\}.
\]
One defines the projective line $\mathbb{P}^1_{\mathbb{F}_1}$ over $\mathbb{F}_1$ by gluing $U_+$ and $U_-$ along $U$. The topological space $U_+$ has two points, the generic point $q_0$ and the closed point $q_+$ containing $t$. Similarly, the topological space $U_-$ has two points, the generic point $r_0$ and the closed point $r_-$ containing $t^{-1}$. Hence, the projective line $\mathbb{P}^1_{\mathbb{F}_1}$ over $\mathbb{F}_1$ consists of three points $\{q_+,p_0,r_-\}$. For the details, we refer the readers to \cite{Deitmar}.

\section{Valuation theory over semirings}\label{valuationtheorysection}
Throughout the paper, we assume all semirings are commutative unless stated otherwise. Recall that an \textit{idempotent} semiring is a semiring $S$ such that $x+x=x$ $\forall x\in S$. An idempotent semiring $S$ has a canonical partial order $\leq$ such that $x \leq y \iff x+y=y$. By a totally ordered idempotent semiring, we always mean a semiring $S$ which is totally ordered with respective to the above canonical order.\\
Theory of valuations on semirings was also considered in \cite{izhakian2011supertropical}(or \cite{izavalue}) by Izhakian, Knebusch, and Rowen. The authors develop a supervaluation theory and recast some basic result (Kapranov's lemma) in terms of a supervaluation theory. In fact, their result is similar to our observation on relations between theory of valuations on semirings and tropical curves in \S \ref{tropicalcurve}.\\ 
In this paper, our novelty lies in considering abstract nonsingular curves in the framework of semirings as well as the implementation of hyperfields into valuation theory on semirings to yield a new concept, \emph{hyperfield valuations}. Now, we provide three different definitions.

\begin{mydef}\label{classcalvaluation}
Let $S$ be a semiring. A \emph{classical valuation} on $S$ is a function $\nu:S \longrightarrow \mathbb{R} \cup \{-\infty\}$ which satisfies the following conditions:
\begin{enumerate}
\item
$\nu(0_S)=-\infty $.
\item
$\nu(xy)=\nu(x)+\nu(y)$, where $+$ is the usual addition of $\mathbb{R}$.
\item
$ \nu(x+y)\leq \max\{\nu(x),\nu(y)\}$ $\forall x,y \in S$.
\item
$\nu(S) \neq \{-\infty\}$. 
\end{enumerate}
\end{mydef}

\begin{rmk}
\begin{enumerate}
\item
We note that one may use the minimum convention, $\min\{\nu(x),\nu(y)\}$, for the third condition which is equivalent to the maximum convention once we reverse the order of $\mathbb{R}$ and change $\{-\infty\}$ to $\{\infty\}$. We use the maximum convention to be consistent with the other two definitions of valuations. 
\item 
The third condition is redundant in some cases. For example, if $S$ is a totally ordered idempotent semiring, then $x+y \in \{x,y\}$ $\forall x,y \in S$ and hence the third condition always holds.
\end{enumerate}
\end{rmk}

\begin{mydef}(\cite[Definition 2.2]{izavalue})\label{semiringmorphism}
Let $S$ be a semiring. A \emph{strict valuation} on $S$ is a function $\nu:S \longrightarrow \mathbb{R}_{max}$ which satisfies the following conditions:
\begin{enumerate}
\item
$\nu(0_S)=-\infty $.
\item
$\nu(xy)=\nu(x)+\nu(y)$, where $+$ is the usual addition of $\mathbb{R}$.
\item
$\nu(x+y)=\max \{\nu(x),\nu(y)\}$ $\forall x,y \in S$.
\item
$\nu(S) \neq \{-\infty\}$. 
\end{enumerate}
In other words, a strict valuation $\nu$ is a non-trivial homomorphism of semirings from $S$ to the tropical semifield $\mathbb{R}_{max}$.
\end{mydef}
\begin{rmk}
Note that $\ker(\nu)=\{0_S\}$ does not imply that $\nu$ is injective.
\end{rmk}
As we mentioned earlier, Definition \ref{semiringmorphism} can be justified in the sense that for a valuation $\nu$ on a commutative ring, we have $\nu(a+b) \in \{\nu(a),\nu(b)\}$ when $\nu(a) \neq \nu(b)$. Classically, the third condition is a subadditivity condition. However, we force the third condition to be an additivity condition. It seems that this is the reason why the authors of \cite{izhakian2011supertropical} named it a strict valuation. One can think of the similar generalization over a hyperfield. To this end, we introduce the following hyperfield.

\begin{mydef}\label{tropicalhyperfield}
The \emph{tropical hyperfield} $\mathbb{T}$ has the underlying set $\mathbb{R} \cup \{-\infty\}$. The addition $\oplus$ is defined as follows: for $x,y \in \mathbb{T}$, 
\[
x\oplus y =\left\{ \begin{array}{ll}
\max\{x,y\} & \textrm{if $x\neq y$}\\
\left[-\infty,x\right]& \textrm{if $x=y$}
\end{array} \right.
\]
The multiplication $\odot$ is given by the usual addition of real numbers with $a\odot(-\infty)=-\infty$ for $\forall a \in \mathbb{R}\cup \{-\infty\}$.
\end{mydef}

\begin{rmk}\label{hyprmk}
The addition of $\mathbb{T}$ is designed to capture the information we lose when $\nu(x)=\nu(y)$ since, in this case, $\nu(x+y)$ can be any number less than or equal to $\nu(x)$. 
\end{rmk}
\begin{pro}
$\mathbb{T}$ is a hyperfield.
\end{pro}
\begin{proof}
A more generalized result is proven in \cite{viro}.
\end{proof}

\noindent Next, we define a valuation of a semiring with values in $\mathbb{T}$. We will call this valuation as a \emph{hyperfield valuation}.

\begin{mydef}\label{semihypval}
Let $S$ be a semiring. A \emph{hyperfield valuation} on $S$ is a function $\nu:S \longrightarrow \mathbb{T}$ such that $\nu(S) \neq \{-\infty\}$ and satisfies the following conditions:
\begin{equation}\label{valuesinH}
\nu(x +y) \in \nu(x) \oplus \nu(y), \textrm{ } \nu(xy)=\nu(x) \odot \nu(y), \textrm{ } \nu(0_S)=-\infty.
\end{equation}

\end{mydef}

\begin{mydef}\label{equivalenceval}
Let $\nu_1$ and $\nu_2$ both be classical valuations (Definition \ref{classcalvaluation}) or strict valuations (Definition \ref{semiringmorphism}) or hyperfield valuations (Definition \ref{semihypval}) on a semiring $S$. We say that $\nu_1$ and $\nu_2$ are \emph{equivalent} if there exists a real number $\rho >0$ such that \[\nu_1(x)=\rho \nu_2(x) \textrm{ }\forall x \in S,\] 
where $\rho \nu_2(x)$ is the usual multiplication of real numbers.
\end{mydef}

\begin{mydef}
Let $\nu$ be any of three valuations on a semiring $S$. We say that $\nu$ is a \emph{discrete valuation} if $\nu(S-\{0_S\}) \simeq \mathbb{Z}$. 
\end{mydef}

\noindent Next, we prove several propositions which are analogous to the classical results. 

\begin{pro}\label{1lem}
Let $S$ be a semiring. Let $\nu$ be any of three valuations on $S$. Then $\nu(1_S)=0$. 
\end{pro}
\begin{proof}
Since $\nu(1_S)=\nu(1_S\cdot 1_S)=\nu(1_S)+\nu(1_S)$, $\nu(1_S)$ should be either $0$ or $-\infty$. However, when $\nu(1_S)=-\infty$, we have $\nu(S)=\{-\infty\}$ and hence $\nu(1_S)=0$. 
\end{proof}

\begin{pro}\label{domain}
Let $S$ be a semifield. Let $\nu$ be any of three valuations on $S$. Then we have
\[
\nu(x)=-\infty \iff x=0_S.
\] 
\end{pro}
\begin{proof}
By the definitions, we have $\nu(0_S)=-\infty$. Conversely, suppose that $\nu(x)=-\infty$ and $x \neq 0_S$. Since $S$ is a semifield, we have $0=\nu(1_S)=\nu(x\cdot x^{-1})=\nu(x)+\nu(x^{-1})$ and this is a contradiction. Therefore, $x=0_S$. 
\end{proof}

\begin{rmk}
Proposition \ref{domain} is valid when $S$ is a multiplicatively cancellative idempotent semiring. The result is obtained by using Proposition \ref{valuefraction}.
\end{rmk}

\begin{pro}\label{valuationring}
Let $S$ be a semiring and $\nu$ be any of three valuations. Let us define the following sets:
\[
R:=\{x \in S \mid \nu(x) \leq 0\} \cup \{0_S\}, \quad\mathfrak{m}:=\{x \in S \mid \nu(x) < 0\} \cup\{0_S\}.
\]
Then $R$ is a subsemiring of $S$ and $\mathfrak{m}$ is an ideal of $R$. We call this $R$ the valuation ring of $\nu$.
\end{pro}
\begin{proof}
Suppose that $\nu$ is a classical valuation. Since $\nu(1_S)=0$, we have $1_S \in R$ and $0_S \in R$ by the definition. For any $x,y \in R$, we have $\nu(xy)=\nu(x)+\nu(y) \leq 0$, thus $xy \in R$. Furthermore, since $\nu(x+y)\leq \max\{\nu(x),\nu(y)\} \leq 0$, we have $x+y \in R$. By the similar argument one can show that $\mathfrak{m}$ is an ideal of $R$. When $\nu$ is either a strict valuation or a hyperfield valuation, one can similarly prove. 
\end{proof}

\begin{cor}\label{equihassamerings}
Let $S$ be a semifield. Let $\nu_1$ and $\nu_2$ be valuations both defined by any of Definitions \ref{classcalvaluation}, \ref{semiringmorphism}, \ref{semihypval}. Let $R_1$ and $R_2$ be the valuation rings of $\nu_1$ and $\nu_2$ respectively. If $\nu_1$ and $\nu_2$ are equivalent then $R_1=R_2$.
\end{cor}
\begin{proof}
This is clear from the definition.
\end{proof}

\begin{cor}\label{cor2}
Let $S$ be a semifield and $\nu$ be a valuation on $S$ (defined by any of Definitions \ref{classcalvaluation}, \ref{semiringmorphism}, \ref{semihypval}). Then an ideal $\mathfrak{m}$ as in Proposition \ref{valuationring} is a unique maximal ideal of the valuation ring $R$. 
\end{cor}
\begin{proof}
In any case, if $\nu(a)=0$ then $a \neq 0$. Therefore, there exists $a^{-1} \in S$. Furthermore, $\nu(a)+\nu(a^{-1})=\nu(aa^{-1})=\nu(1_S)=0$. In other words, $a$ is invertible and $a^{-1} \in R$. Hence $\mathfrak{m}$ is a unique maximal ideal of $R$.
\end{proof}

The following proposition shows that a hyperfield valuation is in fact a generalization of a strict valuation.
\begin{pro}\label{generator}
Let $S$ be a semiring. Suppose that $\nu:S \longrightarrow \mathbb{R}\cup \{-\infty\}$ is a strict valuation. Then $\nu$ is also a hyperfield valuation.
\end{pro}
\begin{proof}
We first note that since strict valuations and hyperfield valuations are functions from $S$ to the set $\mathbb{R} \cup \{-\infty\}$ which satisfy some conditions, the above statement makes sense. Since $\nu$ is a strict valuation and the multiplicative structure of $\mathbb{T}$ is same as that of $\mathbb{R}_{max}$, we have $\nu(xy)=\nu(x)\odot\nu(y)$. We also clearly have $\nu(0_S)=-\infty$ and $\nu(S)\neq \{-\infty\}$. All it remains to show is that $\nu(f+g) \in \nu(f)\oplus \nu(g)$ for any $f,g \in S$. Since $\nu$ is a strict valuation, we have $\nu(f+g)=\max\{\nu(f),\nu(g)\}$. If $\nu(f)=\nu(g)$ then we have $\nu(f+g)=\max\{\nu(f),\nu(g)\} \in \left[-\infty, \nu(f) \right]$. If $\nu(f) \neq \nu(g)$ then the addition of $\mathbb{T}$ agrees with the addition of $\mathbb{R}_{max}$ and hence we have $\nu(f+g) \in \nu(f)\oplus \nu(g)$.
\end{proof}

\section{Examples} \label{more}

\subsection{ The first example, $\mathbb{Q}_{max}$}
\vspace{5pt}
\begin{pro}\label{classQ_max}
Let $S=\mathbb{Q}_{max}$. Then,
\begin{enumerate}
\item
The set of classical valuations on $S$ is equal to $\mathbb{R}$. There are exactly three classical valuations on $S$ up to equivalence.
\item
The set of strict valuations on $S$ is equal to $\mathbb{R}_{\geq 0}$. There are exactly two strict valuations on $S$ up to equivalence.
\item
The set of hyperfield valuations on $S$ is equal to $\mathbb{R}_{\geq 0}$. There are exactly two hyperfield valuations on $S$ up to equivalence.
\end{enumerate}
\end{pro}

\begin{proof}
To avoid any possible confusion, let us denote by $\oplus$, $\odot$ the addition and the multiplication of $S$ respectively. Also the additive and the multiplicative identities of $S$ will be denoted by $0_S$ and $1_S$.
\begin{enumerate}
\item
In this case, as we previously remarked, the third condition is redundant since $S$ is totally ordered. We claim that any valuation $\nu$ on $S$ only depends on the value $\nu(1)$. In fact, since $\mathbb{Z}$ is (multiplicatively) generated by $1$ in $\mathbb{Q}_{max}$, it follows from the second condition that the value $\nu(1)$ determines $\nu(m)$ $\forall m \in \mathbb{Z}$. Moreover, for $\frac{1}{n}$, we have $\nu(1)=\nu(\frac{1}{n}\odot\cdots\odot\frac{1}{n})=n\nu(\frac{1}{n})$ and hence $\nu(\frac{1}{n})=\frac{1}{n}\nu(1)$. This implies that for $\frac{m}{n} \in \mathbb{Q}$, we have $\nu(\frac{m}{n})=\frac{m}{n}\nu(1)$. Conversely, let $\nu:S \longrightarrow \mathbb{R}\cup \{-\infty\}$ be a function such that $\nu(\frac{a}{b}):=\frac{a}{b}\nu(1)$ for some $\nu(1) \neq -\infty$. Then, clearly $\nu$ is a classical valuation on $S$. It follows that the set of classical valuations on $S$ is equal to $\mathbb{R}$.\\ 
Next, suppose that $\nu_1$, $\nu_2$ are classical valuations on $S$ such that $\nu_1(1)>0,\nu_2(1)>0$, then they are equivalent. In fact, if we take $\rho:=\frac{\nu_1(1)}{\nu_2(1)}$, then for $x \in \mathbb{Q}_{max}\backslash\{-\infty\}$, we have $\nu_1(x)=x\nu_1(1)=x\rho\nu_2(1)=\rho\nu_2(x)$. Similarly, classical valuations $\nu_1$ and $\nu_2$ on $S$ with $\nu_i(1) <0$ are equivalent. Finally, $\nu(1)=0$ gives the trivial valuation. Therefore, we have exactly three classical valuations up to equivalence. 
\item
In this case, we claim that a strict valuation $\nu$ is an order-preserving map. Indeed, we have $x \leq y \Longleftrightarrow x\oplus y=y$. Suppose that $x \leq y$. Then we have $\nu(y)=\nu(x\oplus y)=\max\{\nu(x),\nu(y)\} \Longleftrightarrow \nu(x)\leq \nu(y)$. On the other hand, as in the above case, a strict valuation $\nu$ only depends on $\nu(1)$. Since $\nu$ is an order-preserving map and $\nu(0)=0$, it follows that $\nu(1) \geq 0$. Therefore, the set of strict valuations on $S$ is equal to $\mathbb{R}_{\geq 0}$. Moreover, if $\nu(1)=0$, then we have the trivial valuation and strict valuations $\nu$ on $M$ such that $\nu(1) >0$ are equivalent as in the above case. Thus, in this case, there are exactly two strict valuations on $S$ up to equivalence.
\item
We let $\boxplus$ and $\boxdot$ be the addition and the multiplication of $\mathbb{T}$ to avoid any confusion. In this case, a hyperfield valuation $\nu$ on $S$ is order-preserving and determined by $\nu(1)$ and $\nu(1) \geq 0$. In fact, suppose that $x\leq y$. Then we have 
\begin{equation}\label{orderpreserving}
\nu(x\oplus y)=\nu(y) \in \nu(x) \boxplus \nu(y).
\end{equation} 
Assume that $\nu(y) < \nu(x)$. Then we have $\nu(x) \boxplus \nu(y)=\nu(x)$ and it follows from \eqref{orderpreserving} that $\nu(x)=\nu(y)$ which is a contradiction. This shows that $\nu$ is an order-preserving map, where an order on $\mathbb{T}$ is the usual order of $\mathbb{R}$. Furthermore, we have $\nu(0)=0$ since $\nu(0\odot 0 )=\nu(0)=\nu(0)\boxdot \nu(0)=\nu(0)+\nu(0)$. It follows that $\nu(1) \geq 0(=1_{\mathbb{T}})$. Finally, similar to the first case, we have $\nu(\frac{a}{b})=\frac{a}{b}\nu(1)$. Conversely, it is clear that all maps which satisfy such properties are hyperfield valuations on $S$. Hence, the set of hyperfield valuations on $S$ is equal to $\mathbb{R}_{\geq 0}$. Furthermore, two valuations $\nu_1$, $\nu_2$ on $S$ with $\nu_1(1)$, $\nu_2(1) >0$ are equivalent as in the first case. Hence, there are exactly two hyperfield valuations on $S$ up to equivalence.
\end{enumerate}
\end{proof}

\subsection{The second example, $\mathbb{Q}_{max}(T)$}\label{valuationofQmax}\hspace*{\fill} \\
We begin by investigating $\mathbb{Q}_{max}[T]$, the idempotent semiring of polynomials in a variable $T$ with coefficient in $\mathbb{Q}_{max}$. In the sequel, we use the notations $+$ and $\cdot$ for the usual operations of $\mathbb{Q}$. We use the notations $\oplus,\odot$ for the operations of $\mathbb{Q}_{max}[T]$ and $+_t,\cdot_t$ for $\mathbb{Q}_{max}$. \\
For $f(T)=\sum_{i=0}^na_iT^i$, $g(T)=\sum_{i=0}^m b_iT^i \in \mathbb{Q}_{max}[T]$, suppose that $n\leq m$. The addition and the multiplication of $\mathbb{Q}_{max}[T]$ are given as follows:  
\begin{equation}\label{polyadd}
(f\oplus g)(T)=\sum_{i=0}^n \max\{a_i,b_i\}T^i \oplus \sum_{i=n+1}^mb_iT^i,
\end{equation}
\begin{equation}\label{polymulti}
(f\odot g)(T)=\sum_{i=0}^{n+m}(\sum_{r+l=i}a_r\cdot_t b_l)T^i=\sum_{i=0}^{n+m}(\max_{r+l=i}\{a_r+b_l\})T^i.
\end{equation}
Note that we can consider the semifield $\mathbb{Q}_{max}$ as an `algebraically closed' semifield since any polynomial equation with coefficients in $\mathbb{Q}_{max}$ has a (tropical) solution in $\mathbb{Q}_{max}$. However, different from the classical case, any polynomial in $\mathbb{Q}_{max}[T]$ does not have to be factored into linear polynomials. Consider the following example.

\begin{myeg}
Let $P(T)=T^{\odot 2}\oplus T \oplus 3 \in \mathbb{Q}_{max}[T]$. Then, $T=\frac{3}{2}$ is a tropical solution of $P(T)$. Suppose that $T^{\odot 2}\oplus T \oplus 3=(T \oplus a)\odot (T\oplus b)$. Then, we have $(T\oplus a)\odot (T\oplus b)=T^{\odot 2} \oplus\max\{a,b\}\odot T \oplus(a+b)$. Thus, for $P(T)$ to be factored into linear polynomials, we should have $\max\{a,b\}=1$ and $a+b=3$, however this is impossible. Hence, $P(T)$ can not be factored into linear polynomials.
\end{myeg}

\noindent To remedy this issue, in tropical geometry, one imposes a functional equivalence relation on $\mathbb{Q}_{max}[T]$ (cf. \cite{noah}). Recall that polynomials $f(T),g(T) \in \mathbb{Q}_{max}[T]$ are functionally equivalent, denoted by $f(T) \sim g(T)$, if $f(t)=g(t)$ $\forall t \in \mathbb{Q}_{max}$. We further recall that a congruence relation on a semiring is an equivalence relation $\sim$ which is stable under addition and multiplication, i.e., if $x\sim y$ and $a\sim b$, then $(x+a) \sim (y+b)$ and $xa\sim yb$.

\begin{pro}\label{congruencereonsemiring}
For $S=\mathbb{Q}_{max}[T]$, a functional equivalence relation $\sim$ on $S$ is a congruence relation.
\end{pro}
\begin{proof}
Clearly, $\sim$ is an equivalence relation. Suppose that $f(T) \sim g(T)$ and $h(T) \sim q(T)$. Then, we have to show that $f(T)\oplus h(T) \sim g(T)\oplus q(T)$ and $f(T)\odot h(T) \sim g(T)\odot q(T)$.  Let $f(T)=\sum_{i=0}^n a_iT^i$, $g(T)=\sum_{i=0}^m b_iT^i$. It is enough to show that 
\[
(f\oplus g)(x)=f(x)+_tg(x), \quad (f\odot g)(x)=f(x)\cdot_tg(x)\quad \forall x \in \mathbb{Q}_{max}.
\] 
We may assume that $n \leq m$. Then we have
\[(f\oplus g)(T)=\sum_{i=0}^n \max\{a_i,b_i\}T^i \oplus \sum_{i=n+1}^mb_iT^i.\]
For $x \in \mathbb{Q}_{max}$, we have, by letting $a_i=-\infty$ for $i=n+1,...,m$,
\[(f\oplus g)(x)=\max_{i=0,...,m}\{\max\{a_i,b_i\}+ix\}.\]
However, $f(x)=\max_{i=0,...n}\{a_i+ix\}$ and $g(x)=\max_{i=0,...m}\{b_i+ix\}$, thus 
\[f(x)+_tg(x)=\max\{f(x),g(x)\}=\max\{\max_{i=0,...,n}\{a_i+ix\},\max_{i=0,...,m}\{b_i+ix\}\}\]\[=\max_{i=0,...,m}\{\max\{a_i,b_i\}+ix\}=(f\oplus g)(x).\] 
This proves the first part. Next, we have 
\[(f\odot g)(T)=\sum_{i=0}^{n+m}(\sum_{r+l=i}a_rb_l)T^i=\sum_{i=0}^{n+m}(\max_{r+l=i}\{a_r+b_l\})T^i.\]
It follows that for $x \in \mathbb{Q}_{max}$, we have
\[(f\odot g)(x)=\max_{0 \leq i \leq n+m}\{\max_{r+l=i}\{a_r+b_l\}+ix\}=\max_{0 \leq i \leq n+m}\{\max_{r+l=i}\{a_r+rx+b_l+lx\}\}.\]
On the other hand, we have
\[f(x)\cdot_t g(x)=\max_{0\leq i \leq n}\{a_i+ix\}+\max_{0 \leq j \leq m}\{b_j+jx\}.\]
Thus, if $f(x)=a_{i_0}+i_0x$ and $g(x)=b_{j_0}+j_0x$ for some $i_0$ and $j_0$, then we have 
\[f(x)\cdot_t g(x)=(a_{i_0}+i_0x)+(b_{j_0}+j_0x)=(a_{i_0}+b_{j_0})+(i_0+j_0)x.\] 
It follows that $f(x)\cdot_t g(x) \leq (f\odot g)(x)$. But, if $(f\odot g)(x)=(a_{r_0}+r_0x)+(b_{l_0}+l_0x)$, then $(f\odot g)(x) \leq f(x) \cdot_t g(x)$. Hence, $(f\odot g)(x)=f(x)\cdot_t g(x)$.
\end{proof}
 
From Proposition \ref{congruencereonsemiring}, the set $\overline{\mathbb{Q}_{max}[T]}:=\mathbb{Q}_{max}[T]/ \sim$ of equivalence classes is an idempotent semiring. In fact, $\overline{\mathbb{Q}_{max}[T]}$ is a semiring since $\sim$ is a congruence relation. Furthermore, for $f(T) \in \mathbb{Q}_{max}[T]$, we have $f(x)+_tf(x)=f(x)$ $\forall x \in \mathbb{Q}_{max}$. This implies that $f(T)\oplus f(T) \sim f(T)$ and hence $\overline{\mathbb{Q}_{max}[T]}$ is an idempotent semiring. It is known that, for $\overline{\mathbb{Q}_{max}[T]}$, the fundamental theorem of tropical algebra holds, i.e., a polynomial $\overline{P(T)} \in \overline{\mathbb{Q}_{max}[T]}$ can be uniquely factored into linear polynomials in $\overline{\mathbb{Q}_{max}[T]}$ (cf. \cite{bernd1} or \cite{tropicalpoly}). In particular, this implies that the notion of the degree of $\overline{f(T)} \in \overline{\mathbb{Q}_{max}[T]}$ is well defined. Furthermore, $\overline{\mathbb{Q}_{max}[T]}$ does not have any multiplicative zero-divisor. Indeed, suppose that $\overline{f(T)}\cdot \overline{g(T)}=\overline{(fg)(T)} \sim (-\infty)$. Then, for $x \in \mathbb{Q}_{max}$, we have $f(x)\cdot_t g(x) =f(x)+g(x)=-\infty$. In other words, for $x \in \mathbb{Q}_{max}$, we have $f(x)=-\infty$ or $g(x)=-\infty$. However, this only happens when $f(T)=-\infty$ or $g(T)=-\infty$. Thus, $\overline{\mathbb{Q}_{max}[T]}$ does not have a multiplicative zero-divisor. In fact, in Corollary \ref{multiplicatcancellative}, we shall prove that $\overline{\mathbb{Q}_{max}[T]}$ satisfies the stronger condition: $\overline{\mathbb{Q}_{max}[T]}$ is multiplicatively cancellative.

\begin{mydef} \label{fntfid}
We define $\mathbb{Q}_{max}(T)$ to be the localization of $\overline{\mathbb{Q}_{max}[T]}$ at $\overline{\mathbb{Q}_{max}[T]}\backslash \{-\infty\}$. In other words, $\mathbb{Q}_{max}(T):=\Frac(\overline{\mathbb{Q}_{max}[T]})$.
\end{mydef}

We first prove several lemmas to classify valuations on $\mathbb{Q}_{max}(T)$.

\begin{lem}\label{v(T)<0}
Let $S:=\overline{\mathbb{Q}_{max}[T]}$. For $\overline{f(T)} \in S$, let $r_f$ be the maximum natural number such that $\overline{T}^{r_f}$ can divide $\overline{f(T)}$. Then, for $\overline{f(T)},\overline{g(T)} \in S$, we have 
\[r_{f\oplus g}= \min\{ r_f, r_g\}, \quad r_{f \odot g}=r_f+r_g.\]
\end{lem}
\begin{proof}
Let $f(T),g(T) \in \mathbb{Q}_{max}[T]$. We first claim that if $f(T)$ has a constant term different from $-\infty$ and $g(T)$ does not have a constant term, then $f(T)$ and $g(T)$ are not functionally equivalent. Indeed, if $f(T)=\sum a_iT^i$ and $g(T)=\sum b_iT^i$, then $f(-\infty)=a_0 \neq -\infty=g(-\infty)$. One can further observe that if $f(T) \sim T$, then $f(T)=T$. In fact, from the fundamental theorem of tropical algebra, we know that the degree of $f(T)$ should be one. Hence, $f(T)=a\odot T \oplus b$ for some $a,b \in \mathbb{Q}_{max}$. Then $b=-\infty$ since, otherwise, $f(-\infty)=b \neq -\infty$ and therefore $f(T) \not \sim T$. Furthermore, $a=0$ since, otherwise, we have $f(-a)=0$. However, this is different from the evaluation of $T$ at $-a$.\\ 
Next, we claim that $\overline{f(T)} \in S$ has the factor $\overline{T}$ if and only if any representative of $\overline{f(T)}$ does not have a constant term. To see this, suppose that $\overline{f(T)}$ has the factor $\overline{T}$. Then, $f(T) \sim T\odot h(T)$ for some $h(T) \in \mathbb{Q}_{max}[T]$. Since $T\odot h(T)$ does not have a constant term, from the first claim, $f(T)$ also does not have a constant term. Conversely, suppose that any representative of $\overline{f(T)}$ does not have a constant term. We can write $f(T)=T \odot h(T)$ for some $h(T) \in \mathbb{Q}_{max}$. Hence, $\overline{f(T)}$ has a factor $\overline{T}$. From the above argument and the fundamental theorem of tropical algebra, it is now clear that $r_f$ is well defined. Moreover, for $\overline{f(T)}, \overline{g(T)} \in S$, we can write $\overline{f(T)}=\overline{T}^l \odot \overline{h(T)}$, $\overline{g(T)}=\overline{T}^m \odot \overline{p(T)}$ for some $\overline{h(T)}$, $\overline{p(T)}$ such that $\overline{h(T)}$ and $\overline{p(T)}$ do not have $\overline{T}$ as a factor. From our previous claim, this is equivalent to that $\overline{h(T)}$ and $\overline{p(T)}$ do have a constant term different from $-\infty$. Assume that $l \leq m$, then we have
\[\overline{f(T)}\oplus \overline{g(T)}=\overline{T}^l\odot (\overline{h(T)} \oplus \overline{T}^{(m-l)}\overline{p(T)}).\]
Since $\overline{h(T)}$ has a constant term different from $-\infty$, it follows that $\overline{h(T)} \oplus \overline{T}^{(m-l)}\overline{p(T)}$ has a constant term different from $-\infty$ and therefore $\overline{h(T)} \oplus \overline{T}^{(m-l)}\overline{p(T)}$ does not have a factor $\overline{T}$. This shows that $r_{f\oplus g}=\min\{r_f,r_g\}$. The second assertion $r_{f\odot g}=r_f+r_g$ is clear from the fundamental theorem of tropical algebra.
\end{proof} 

\begin{rmk}
Lemma \ref{v(T)<0} is different from the classical case. Essentially, this is due to the absence of additive inverses. In the classical case, if $f(T)=T^lh(T)$, $g(T)=T^mp(T) \in \mathbb{Q}[T]$ with $l <m$, then $f(T)+g(T)=T^l(h(T)+T^{(m-l)}p(T))$. Hence, we have $r_{f+g}=\min\{r_f,r_g\}$. The problem is when $l=m$. For example, if $f(T)=T(T+1)$, $g(T)=T(T-1) \in \mathbb{Q}[T]$, then $r_f=r_g=1$. However, $f(T)+g(T)=2T^2$ and hence $r_{f+g}=2 >\min\{r_f,r_g\}=1$ from the additive cancellation which is impossible in the case of idempotent semirings.
\end{rmk}

When one replace $\overline{T}$ with any linear polynomial, one has the following. 

\begin{lem}\label{g_blemma}
Let $S:=\overline{\mathbb{Q}_{max}[T]}$ and $g_b:=\overline{(0 \oplus b\odot T)}$, where $b \in \mathbb{Q}$. For $\overline{f(T)} \in S$, let $r_b(f)$ be the maximum natural number such that $\overline{g_b}^{r_b(f)}$ can divide $\overline{f(T)}$. Then, for $\overline{f(T)},\overline{g(T)} \in S$, we have 
\[\min\{ r_b(f), r_b (g)\} \leq r_b({f\oplus g}), \quad r_b({f \odot g})=r_b(f)+r_b(g).\]
\end{lem}
\begin{proof}
For the notational convenience, we let $r_b(f)=r_f$. From the fundamental theorem of tropical algebra, $r_f$ is well defined and clearly we have $r_{f\odot g}=r_f+r_g$. Let $\overline{f}(T)=g_b^{r_f}\overline{f'(T)}$ and $\overline{g}(T)=g_b^{r_g}\overline{g'(T)}$ for some $\overline{f'(T)}$ and $\overline{g'(T)} \in S$. Suppose that $r_f \leq r_g$. This implies that 
\[
\overline{f(T)\oplus g(T)}=g_b^{r_f}(\overline{f'(T)}\oplus g_b^{(r_g-r_f)}\overline{g'(T)})
\]
and hence $\min\{r_f,r_g\} \leq r_{f\oplus g}$. 
\end{proof}

\begin{rmk}
Lemma \ref{g_blemma} cannot be as sharp as Lemma \ref{v(T)<0}. For instance, let $g_b=(0\oplus 0\odot T)$, $f(T)=(0\oplus 0\odot T)^2\odot T$, and $g(T)=(0\oplus 0\odot T)^2$. Then $r_f=r_g=2$. But, we have
\[
 f(T)\oplus g(T)=(0\oplus 0\odot T)^3
\]
and hence $r_{f\oplus g}=3$. 
\end{rmk}

\begin{lem}\label{degree}
Let $S:=\overline{\mathbb{Q}_{max}[T]}$. Then, for $\overline{f(T)} \in S$, $\deg \overline{f(T)}$ is well defined. Furthermore, for $\overline{f(T)}$, $\overline{g(T)} \in S$, we have 
\[\deg (\overline{f(T)} \oplus \overline{g(T)}) = \max\{\deg \overline{f(T)},\deg \overline{g(T)}\},\] \[  \deg (\overline{f(T)} \odot \overline{g(T)}) = \deg \overline{f(T)}+\deg \overline{g(T)}.\]
\end{lem}
\begin{proof}
This is straightforward from the fundamental theorem of tropical algebra and the fact that no additive cancellation happens when we add two tropical polynomials.
\end{proof}

\begin{cor}\label{multiplicatcancellative}
Let $S:=\overline{\mathbb{Q}_{max}[T]}$. Then $S$ is multiplicatively cancellative. 
\end{cor}
\begin{proof}
For $\overline{f(T)\odot h(T)}=\overline{g(T)\odot h(T)}$ with $h(T)\neq -\infty$, we have to show that $\overline{f(T)}=\overline{g(T)}$. We continue using the notation as in Lemma \ref{v(T)<0}. We know that $\overline{f(T)\odot h(T)}=\overline{g(T)\odot h(T)}$ is equivalent to the following condition: 
\begin{equation}\label{multiplicativefree}
f(x)+h(x)=g(x)+h(x) \quad \forall x \in \mathbb{Q}_{max},
\end{equation}
where $+$ is the usual addition. Thus, if $h(x) \neq -\infty$, we have $f(x)=g(x)$. Since $h(x)=-\infty$ happens only when $x=-\infty$, it follows that $f(x)=g(x)$ as long as $x\neq -\infty$. Hence, all we have to show is that $f(-\infty)=g(-\infty)$. From Lemma \ref{v(T)<0}, we have $r_f+r_h=r_g+r_h$ and therefore $r_f=r_g$. Fix a representative $f(T)=\sum a_iT^i$ of $\overline{f(T)}$. We then have $f(-\infty)=a_0$ if $r_f=0$ and $f(-\infty)=-\infty$ if $r_f \neq 0$. Thus, we may assume that $r_f=r_g=0$. Fix a representative $g(T)=\sum b_iT^i$ of $\overline{g(T)}$. From Lemma $3.2$ of \cite{tropicalpoly}, there exists a real number $N$ such that if $x>N$, then $f(x)=a_0$ and $g(x)=b_0$. Since we know that $f(T)$ and $g(T)$ agree on all elements of $\mathbb{Q}_{max}$ but $-\infty$, we conclude that $f(x)=a_0=b_0=g(x)$ for $x>N$. Therefore, we have $f(-\infty)=a_0=b_0=g(-\infty)$ and hence $\overline{f(T)}=\overline{g(T)}$.
\end{proof}

Let $M:=\overline{\mathbb{Q}_{max}[T]}$, $S:=M \backslash \{\overline{-\infty}\}$, and $\mathbb{Q}_{max}(T)=S^{-1}M$ as in Definition \ref{fntfid}. It follows from Corollary \ref{multiplicatcancellative} that the localization map $S^{-1}:M \longrightarrow S^{-1}M$ is injective and $\mathbb{Q}_{max}(T)$ is an idempotent semifield.

\begin{pro}\label{valuefraction}
Let $M$ be a multiplicatively cancellative idempotent semiring. Let $S:=M\backslash\{0_M\}$ and $N:=S^{-1}M$, the localization of $M$ at $S$. Let $\nu$ be any of three valuations (classical, strict, hyperfield) on $N$. Then, a valuation $\nu$ on $N$ only depends on the image $i(M)$ of the canonical injection $i:M\longrightarrow S^{-1}M=N$, $m \mapsto \frac{m}{1}$.
\end{pro}
\begin{proof}
Since $i$ is an injection, one can identify an element $m \in M$ with $\frac{m}{1} \in S^{-1}M=N$ under the canonical injection $i$. We have $1_N=\frac{a}{a}$ $\forall a \in S=M^{\times}$. Then, with Definitions \ref{classcalvaluation}, \ref{semiringmorphism}, and \ref{semihypval}, we have $\nu(1_N)=\nu(a)+\nu(\frac{1}{a})=0$, where $+$ is the usual addition of real numbers. It follows that $\nu(\frac{1}{a})=-\nu(a)$ and hence $\nu(\frac{a}{b})=\nu(a)-\nu(b)$. 
\end{proof}

\begin{rmk}
In the theory of commutative rings, to be multiplicatively cancellative and to have no (multiplicative) zero divisors are equivalent conditions. This is in contrast to the theory of semirings, where the first condition implies the second condition but not conversely in general. However, even when $M$ is only a semiring without (multiplicative) zero divisors, one can derive the statement as in Proposition \ref{valuefraction} in the following sense: Let $M$ be a semiring without (multiplicative) zero divisors and $Val(M)$ be the set of strict (resp. hyperfield) valuations on $M$. Then, there exists a set bijection between $Val(M)$ and the set of strict (resp. hyperfield) valuations on $Val(S^{-1}M)$. Indeed, for $\nu \in Val(M)$, one can define a valuation $\tilde{\nu} \in Val(S^{-1}M)$ such that $\tilde{\nu}(\frac{a}{b})=\nu(a)\nu(b)^{-1}$. Conversely, for $\nu \in Val(S^{-1}M)$, we define $\hat{\nu}=\nu\circ i \in Val(M)$, where $i:M\longrightarrow S^{-1}M$. One can easily check that these two are well defined and inverses to each other. 
\end{rmk}

\begin{pro}\label{semiringhomvalfunctionfield}
Let $M:=\overline{\mathbb{Q}_{max}[T]}$, $S:=M\backslash \{\overline{-\infty}\}$, and $\mathbb{Q}_{max}(T):=S^{-1}M$. Then, the set of strict valuations on $\mathbb{Q}_{max}(T)$ which are trivial on $\mathbb{Q}_{max}$ is equal to $\mathbb{R}$. There are exactly three strict valuations on $\mathbb{Q}_{max}(T)$ which are trivial on $\mathbb{Q}_{max}$ up to equivalence.
\end{pro}
\begin{proof}
From Proposition \ref{valuefraction} and Corollary \ref{multiplicatcancellative}, a strict valuation $\nu$ on $\mathbb{Q}_{max}(T)$ only depends on values of $\nu$ on $M$. Let $\overline{f(T)} \in M\backslash \{-\infty\}$. Then, from the fundamental theorem of tropical algebra, we have \[\overline{f(T)}=\overline{l_0(T)}\odot\overline{l_1(T)}\odot \overline{l_2(T)}\odot...\odot \overline{l_n(T)},\] 
where $l_0(T) \in \mathbb{Q}$ and $l_i(T)=a_iT\oplus b_i$ for some $a_i\in \mathbb{Q},b_i \in \mathbb{Q}_{max}$. Since we are considering valuations which are trivial on $\mathbb{Q}_{max}$, we may assume that $l_0(T)=1_{\mathbb{R}_{max}}=0$. It follows that \[\nu(\overline{f(T)})=\nu(\overline{l_1(T)}) +
\nu(\overline{l_2(T)})+...+ \nu(\overline{l_n(T)}).\]
Let us first assume that $\nu(\overline{T}) <0$. If $b_i \neq -\infty$, since $\nu$ is trivial on $\mathbb{Q}_{max}$, we have 
\[\nu(a_i\overline{T}\oplus b_i)=\max\{(\nu(a_i)+\nu(\overline{T})),\nu(b_i)\}
=\max\{\nu(\overline{T}),0\}=0.
\]
Thus, we have 
\begin{equation}\label{strctva}
\nu(\overline{f(T)})=r_f(\nu(\overline{T})),
\end{equation} 
where $r_f$ is as in Lemma \ref{v(T)<0}. Conversely, any map $\nu:\mathbb{Q}_{max}(T) \longrightarrow \mathbb{R}_{max}$ satisfying the following conditions
\[\nu(q)=0 \quad \forall q \in \mathbb{Q},\quad  \nu(-\infty)=-\infty,\quad  \nu(\overline{T}) <0,\quad \nu(\overline{f(T)})=r_f(\nu(\overline{T}))\] is indeed a strict valuation. In fact, from Lemma \ref{v(T)<0}, we know that $r_{f \oplus g}=\min\{r_f,r_g\}$. Since $\nu(\overline{T}) <0$ and $r_f,r_g \in \mathbb{N}$, this implies that 
\[\nu(\overline{f(T)} \oplus \overline{g(T)})=\nu(\overline{f(T)\oplus g(T)})=r_{f \oplus g}\nu(\overline{T})=\min\{r_f,r_g\}\nu(\overline{T})\]\[=\max\{r_f\nu(\overline{T}),r_g\nu(\overline{T})\}
=\max\{\nu(\overline{f(T)}),\nu(\overline{g(T)})\}.\]
Moreover, $\nu(\overline{f(T)}\odot \overline{g(T)})=\nu(\overline{f(T)\odot g(T)})=r_{f\odot g}\nu(\overline{T})=(r_f+r_g)\nu(\overline{T})=r_f\nu(\overline{T})+r_g\nu(\overline{T})=\nu(\overline{f(T)})+\nu(\overline{g(T)})$.
Furthermore, all such strict valuations on $\mathbb{Q}_{max}(T)$ are equivalent. Indeed, let $\nu_1, \nu_2$ be strict valuations on $\mathbb{Q}_{max}(T)$ such that $\nu_1(\overline{T})=\alpha<0$ and $\nu_2(\overline{T})=\beta<0$.
Since $\alpha, \beta$ are negative numbers, $\rho:=\frac{\beta}{\alpha}$ is a positive number and  $\nu_2(\overline{f(T)})=r_f\beta=(r_f\rho)\alpha=\rho\nu_1(\overline{f(T)})$.\\
Secondly, suppose that $\nu(\overline{T})=0$. Then, we have
\[\nu(a_i\overline{T}\oplus b_i)=0.\]
In other words,  $\nu$ is a trivial valuation since $0=1_{\mathbb{R}_{max}}$. Clearly, this is not equivalent to the first case.\\
The final case is when $\nu(\overline{T}) >0$. Then we have 
\[\nu(a_i\overline{T}\oplus b_i)=\max\{(\nu(a_i)+\nu(\overline{T})),\nu(b_i)\}
=\max\{\nu(\overline{T}),0\}=\nu(\overline{T}).
\]
It follows that
\[\nu(\overline{f(T)})=\deg(\overline{f(T)})\nu(\overline{T}).\] 
Conversely, any map $\nu:\mathbb{Q}_{max}(T) \longrightarrow \mathbb{R}_{max}$ satisfying the following conditions
\[\nu(q)=0 \quad \forall q \in \mathbb{Q},\quad  \nu(-\infty)=-\infty,\quad  \nu(\overline{T})>0,\quad \nu(\overline{f(T)})=\deg(\overline{f(T)})(\nu(\overline{T}))\] is indeed a strict valuation from Lemma \ref{degree}.
Furthermore, suppose that $\nu_1$, $\nu_2$ are strict valuations on $\mathbb{Q}_{max}(T)$ such that $\nu_1(\overline{f(T)})=\alpha>0, \nu_2(\overline{f(T)})=\beta>0$. Then, with $\rho=\frac{\beta}{\alpha}$, $\nu_1$ and $\nu_2$ are equivalent. Furthermore, this case is not equivalent to any of the above. In summary, the set of strict valuations on $\mathbb{Q}_{max}(T)$ which are trivial on $\mathbb{Q}_{max}$ is equal to $\mathbb{R}$ (by sending $\nu$ to $\nu(\overline{T})$). There are exactly three strict valuations depending on a sign of a value of $\overline{T}$.
\end{proof}

Next, we compute the set of hyperfield valuations on $\mathbb{Q}_{max}(T)$ which are trivial on $\mathbb{Q}_{max}$. In what follows, to avoid notational confusion, we denote by $\oplus, \odot$ the addition and the multiplication of idempotent semirings, by $+,\cdot$ the usual addition and multiplication of $\mathbb{Q}$, and by $\boxplus, \boxdot$ the addition and the multiplication of $\mathbb{T}$.

\begin{lem}\label{primelemma}
Let $M:=\overline{\mathbb{Q}_{max}[T]}$ and $\nu:M\to \mathbb{T}$ be a hyperfield valuation which is not trivial. Suppose that $\nu$ is trivial on $\mathbb{Q}_{max}$ and $\nu(\overline{T})=0$. Let
\[
\mathfrak{p}_\nu:=\{f \in M \mid \nu(f) <0\}.
\]
Then $\mathfrak{p}_\nu$ is a prime ideal. 
\end{lem}
\begin{proof}
For any $f,g \in \mathfrak{p}_\nu$, we have $\nu(f\oplus g) \in \nu(f) \boxplus \nu(g)$. But, since $\nu(f)$ and $\nu(g)$ are negative, $\nu(f\oplus g) <0$ and hence $f \oplus g \in \mathfrak{p}_\nu$. Next, let $h \in M$. We claim that if $h \not \in \mathfrak{p}_\nu$ then $\nu(h)=0$. Indeed, we may assume that $h=a\oplus  b\odot \overline{T}$. Hence, $\nu(h)=\nu(a\oplus b \odot \overline{T})$. Since $\nu(\overline{T})=0$ and $\nu$ is trivial on $\mathbb{Q}_{max}$, we have $\nu(h) \in 0 \boxplus (0 \boxdot 0)$, i.e., $\nu(h) \leq 0$. Since $h \not \in \mathfrak{p}_\nu$, this implies that $\nu(h)=0$. Now, if $h \not \in \mathfrak{p}_\nu$ and $f \in \mathfrak{p}_\nu$, then $\nu(h\odot f)=\nu(h)\boxdot \nu(f)=0+\nu(f)=\nu(f) <0$ and hence $h\odot f \in \mathfrak{p}_\nu$. If both $f,h \in \mathfrak{p}_\nu$, then clearly $h\odot f \in \mathfrak{p}_\nu$. Finally, suppose that $f \odot g \in \mathfrak{p}_\nu$. In other words, $\nu(f \odot g) <0$. This implies that $\nu(f) <0$ or $\nu(g) <0$ since $\nu(f\odot g)=\nu(f)+\nu(g)$. This proves that $\mathfrak{p}_\nu$ is a prime ideal. 
\end{proof}

\begin{lem}\label{lemma1}
Let $M:=\overline{\mathbb{Q}_{max}[T]}$ and $\nu:M\to \mathbb{T}$ be a hyperfield valuation which is not trivial, but $\nu$ is trivial on $\mathbb{Q}_{max}$ and $\nu(\overline{T})=0$. Suppose that $\mathfrak{p}_\nu$ is generated by $g=a\oplus b\odot \overline{T}$ with $a \neq -\infty$. Then $\nu$ is uniquely determined by $\nu(g)$. 
\end{lem}
\begin{proof}
Suppose that $h \not \in \mathfrak{p}_\nu$. In Lemma \ref{primelemma}, we proved that in this case  $\nu(h)=0$. Now, suppose that $f \in \mathfrak{p}_\nu-\{-\infty\}$. Then we can write $f=b\odot g^{\odot n}\odot h$ with $b \in \mathbb{Q}_{max}$ and $h \not \in \mathfrak{p}_\nu$. This forces $\nu(f)=n\nu(g)$ since $\nu(b)=\nu(h)=0$. 
\end{proof}

\begin{lem}\label{lemma2}
Let $M:=\overline{\mathbb{Q}_{max}[T]}$, $t$ a negative real number, and $g:=a\oplus b\odot \overline{T} \in M$ with $a \neq -\infty$. Let $\nu:M \to \mathbb{T}$ be a function such $\nu(f)=nt$, where $n$ is the largest natural number such that $g^n$ divides $f$ and $\nu(0_M)=-\infty$. Then $\nu$ is a hyperfield valuation.  
\end{lem}
\begin{proof}
From the fundamental theorem of tropical algebra, we have $\nu(f\odot g)=\nu(f)\boxdot \nu(g)$. Now, let $f=g^xf'$ and $h=g^yh'$ such that $g$ can not divide both $f'$ and $h'$. From Lemma \ref{g_blemma}, we have that $\min\{r_f,r_g\} \leq r_{f\oplus g}$ and since $t <0$,
\[
t\cdot (r_{f \oplus g}) \leq \min\{x,y\}\cdot t=\max\{x\cdot t,y\cdot t\} \iff \nu(f\oplus g) \leq \max \{\nu(f),\nu(g)\}. 
\]
It follows that $\nu(f\oplus g) \in \nu(f)\boxplus \nu(g)$. This completes the proof. 
\end{proof}

\begin{cor}\label{equicor}
Let $M:=\overline{\mathbb{Q}_{max}[T]}$. Suppose that $\nu_1$ and $\nu_2$ are hyperfield valuations on $M$ which is trivial on $\mathbb{Q}_{max}$ and $\nu_i(\overline{T})=0$ for $i=1,2$. If $\mathfrak{p}_{\nu_i}$ is generated by $g \in M$ for both $i=1,2$, then $\nu_1$ and $\nu_2$ are equivalent. 
\end{cor}
\begin{proof}
This is clear from the proof of Lemma \ref{lemma1}.  
\end{proof}

\begin{lem} \label{principlelemma}
Let $M:=\overline{\mathbb{Q}_{max}[T]}$. Let $\nu:M\to\mathbb{T}$ be a hyperfield valuation on $M$ which is trivial on $\mathbb{Q}_{max}$ and $\nu(\overline{T})=0$. Then $\mathfrak{p}_\nu=<g>$ for some linear polynomial $g \in M$. 
\end{lem}
\begin{proof}
Since any element $f$ in $M$ can be uniquely factored into linear polynomials, it is enough to show that any two linear polynomials in $\mathfrak{p}_\nu$ is an non-zero constant multiple of each other. Let $f=a\oplus b\odot \overline{T}$ and $h=x \oplus y \odot \overline{T}$ be elements of $\mathfrak{p}_\nu$. In particular, $a,x \neq -\infty$ since $\nu(\overline{T})=0$, but $\nu(f),\nu(h) <0$. We can write $f=t\odot h \oplus q$ for some $t,q \in \mathbb{Q}_{\max}$ and $t \neq -\infty$. If $q=-\infty$, then there is nothing to prove. Suppose that $q \neq -\infty$. Then we have, since $\nu(h) <0$, 
\[
\nu(f)=\nu(t\odot h \oplus q) \in \nu(t)\boxdot \nu(h) \boxplus \nu(q)=\nu(h) \boxplus 0 =0
\] 
This is absurd since $\nu(f) <0$ and hence $q=-\infty$. This proves that $f$ is a non-zero constant multiple of $h$. 
\end{proof}

\begin{lem}\label{onlyonelem}
Let $M:=\overline{\mathbb{Q}_{max}[T]}$ and $g_b:=0\oplus b\odot \overline{T}$ for $b \in \mathbb{Q}$. If $c \in \mathbb{Q}$ and $g_c=r\odot g_b$ for some $r \in \mathbb{Q}_{max}$, then $r=0$. In other words, $b=c$.  
\end{lem}
\begin{proof}
If $b=c$, then there is nothing to prove. Suppose that $b \neq c$. We may assume that $c <b$. Since $g_c=r\odot g_b$, we have
\begin{equation}\label{eqeq}
\max\{0,c+t\}=r+\max\{0,b+t\}=\max\{r,r+b+t\}, \quad \forall t \in \mathbb{Q}_{max}.
\end{equation} 
Let $t=-c$. Then \eqref{eqeq} becomes
\[
\max\{0,0\}=0=\max\{r,r+(b-c)\}=r+(b-c)
\]
This implies that $r=c-b$. Now, when $t=-b$, \eqref{eqeq} becomes
\[
\max\{0,c-b\}=0=\max\{r,r+0\}=r. 
\]  
This implies that $r=0=c-b$ and hence $c=b$. 
\end{proof}

\begin{pro}\label{valR_valfunctionfield}
Let $M:=\overline{\mathbb{Q}_{max}[T]}$, $S:=M \backslash \{\overline{-\infty}\}$, and $\mathbb{Q}_{max}(T):=S^{-1}M$. Then, the set of the equivalence classes of non-trivial hyperfield valuations on $\mathbb{Q}_{max}(T)$ which are trivial on $\mathbb{Q}_{max}$ is equal to $\mathbb{Q}\cup \{-\infty\} \cup \{\infty\}$. 
\end{pro}
\begin{proof}
%To avoid notational confusion, we denote by $\oplus, \odot$ the addition and the multiplication of idempotent semirings and by $\vee, \cdot$ the addition and the multiplication of $\mathbb{T}$.
From Proposition \ref{valuefraction}, a hyperfield valuation $\nu$ on $\mathbb{Q}_{max}(T)$ only depends on values of $\nu$ on $M$ and $\nu(M-\{0_M\})$ should be a subset of $\mathbb{R}$. Let $\nu$ be a hyperfield valuation on $\mathbb{Q}_{max}(T)$ which is trivial on $\mathbb{Q}_{max}$. For $\overline{f(T)} \in M\backslash \{-\infty\}$, from the fundamental theorem of tropical algebra, we have 
\[\overline{f(T)}=\overline{l_0(T)}\odot\overline{l_1(T)}\odot \overline{l_2(T)}\odot...\odot \overline{l_n(T)},\] 
where $l_0(T) \in \mathbb{Q}$ and $l_i(T)=a_i\odot T\oplus b_i$ for some $a_i\in \mathbb{Q}, b_i \in \mathbb{Q}_{max}$. As in Proposition \ref{semiringhomvalfunctionfield} we may assume that $l_0(T)=0$. Hence, $\nu$ is entirely determined by its values on linear polynomials. Similar to Proposition \ref{semiringhomvalfunctionfield}, we divide the cases up to a sign of $\nu(\overline{T})$. The first case is when $\nu(\overline{T})<0$. Since $\nu$ is trivial on $\mathbb{Q}_{max}$, if $b \neq -\infty$, we have
\[\nu(a\odot \overline{T}\oplus b) \in \nu(a\odot \overline{T})\boxplus  \nu(b)=(\nu(a)\boxdot \nu(\overline{T}))\boxplus \nu(b)=\max\{\nu(\overline{T}),0\}=0.\]
With the same notation as in Lemma \ref{v(T)<0}, we have \begin{equation}\label{valucondition}
\nu(\overline{f(T)})=r_f \nu(\overline{T}).
\end{equation}
Conversely, any map $\nu:\mathbb{Q}_{max}(T) \longrightarrow \mathbb{T}$ given by \eqref{valucondition} is a hyperfield valuation when $\nu(\overline{T})<0$. Indeed, from the fundamental theorem of tropical algebra, we have \[\nu(\overline{f(T)}\odot \overline{g(T)})=(r_f+r_g)\nu(\overline{T})=r_f\nu(\overline{T})+r_g\nu(\overline{T})=\nu(\overline{f(T)}) \boxdot \nu(\overline{g(T)}).
\] 
Moreover, from Lemma \ref{v(T)<0}, we have \[\nu(\overline{f(T)} \oplus \overline{g(T)})=r_{(f\oplus g)} \nu(\overline{T})=\min\{r_f,r_g\}\nu(\overline{T})=\max\{r_f\nu(\overline{T}),r_g\nu(\overline{T})\}
\]
\[=\max\{\nu(\overline{f(T)}),\nu(\overline{g(T)})\} \in \nu(\overline{f(T)}) \boxplus \nu(\overline{g(T)}).\] Similar to Proposition \ref{semiringhomvalfunctionfield}, all these cases are equivalent.\\ 
The second case is when $\nu(\overline{T})=0$. Suppose first that $\mathfrak{p}_\nu =\emptyset$ then $\nu$ is a trivial valuation. Indeed, for any $f=a\oplus b\odot \overline{T}$, we have 
\[
\nu(f)=\nu(a\oplus b \odot \overline{T}) \in \nu(a)\boxplus \nu(b)\boxdot \nu(\overline{T})=0\boxplus 0.
\]
This implies that $\nu(f) \leq 0$, however, since $\mathfrak{p}_\nu=\emptyset$, we conclude that $\nu(f)=0$.\\
Now, suppose that $\mathfrak{p}_\nu$ is not empty. It follows from Lemmas \ref{primelemma} and \ref{principlelemma} that $\mathfrak{p}_\nu$ is generated by a single linear polynomial $g=0\oplus b\odot \overline{T} \in M$, where $b \in \mathbb{Q}_{max}-\{-\infty\}$. Furthermore, it follows from Lemma \ref{onlyonelem} that $b$ is uniquely determined. Finally, from Corollary \ref{equicor}, any two such hyperfield valuations $\nu_1$ and $\nu_2$ are equivalent if $\mathfrak{p}_{\nu_1}=\mathfrak{p}_{\nu_2}$. This implies that there exists one-to-one correspondence between the set of the equivalence classes in this case and $\mathbb{Q}$. Clearly, any of this is not equivalent to the first case.\\ 
The final case is when $\nu(\overline{T})>0$. Then, as in Proposition \ref{semiringhomvalfunctionfield}, we have $\nu(\overline{f(T)})=\deg (\overline{f(T)})(\nu(\bar{T}))$. Conversely, any map $\nu:\mathbb{Q}_{max}(T) \longrightarrow \mathbb{T}$ given in this way is indeed a hyperfield valuation by Lemma \ref{degree}. These are all equivalent from the exact same argument in Proposition \ref{semiringhomvalfunctionfield}.\\
To summarize, we can identify the equivalence class of the case $\nu(\overline{T})>0$ to $\infty$ and the equivalence class of the case $\nu(\overline{T})<0$ to $-\infty$ and the equivalence classes with $\mathfrak{p}_\nu \neq \emptyset$ and $\overline{T}=0$ to $\mathbb{Q}$. Then we obtain our desired bijection from the set of equivalence classes of non-trivial hyperfield valuations on $\mathbb{Q}_{max}(T)$ which are trivial on $\mathbb{Q}_{max}$ to the set $\mathbb{Q} \cup\{-\infty\}\cup \{\infty\}$. 
\end{proof}

\begin{rmk}
In Propositions \ref{semiringhomvalfunctionfield} and \ref{valR_valfunctionfield}, the value groups of non-trivial valuations are all isomorphic to $\mathbb{Z}$ and hence they are all discrete valuations. 
\end{rmk}

\begin{rmk}
One may notice that Proposition \ref{valR_valfunctionfield} is very similar to the Berkovich affine line over a trivially valued field. 
\end{rmk}

\section{Geometric aspects of valuations}
In this section, we consider the abstract curve associated to the function field $\mathbb{Q}_{max}(T)$ by using strict valuations as well as hyperfield valuations. To be precise, the abstract curve, as the set of strict valuations on $\mathbb{Q}_{max}(T)$, yields the projective line $\mathbb{P}^1_{\mathbb{F}_1}$ over the field with one element. On the other hand, when we employ hyperfield valuations, we recast the tropical projective line over $\mathbb{Q}_{max}$, where we consider the tropical projective line over $\mathbb{Q}_{max}$ as the set $\mathbb{Q}_{max}\cup \{\infty\}$ which is homeomorphic to $[0,1] \cap \mathbb{Q}$ with the Euclidean topology. \\

Let us recall the definition of the abstract nonsingular curve $C_K$ associated to a function field $K$ over an algebraically closed field $k$. As a set $C_K$ consists of equivalence classes of discrete valuations on $K$ which are trivial on $k$. Equivalently, $C_K$ can be considered as the set of all discrete valuation rings of $K$ that are trivially valued on $k$. Note that the similar results for idempotent semifields are proved in Corollary \ref{equihassamerings}.\\
Classically, one imposes a topology on $C_K$ by defining the closed sets to be the finite subsets and $C_K$. The structure sheaf $\mathcal{O}_{C_K}$ on $C_K$ is defined as follows: for each open subset $U$ of $C_K$, 
\begin{equation}\label{sheaf}
\mathcal{O}_{C_K}(U):= \bigcap_{R \in U} R, 
\end{equation}
where $R$ is a discrete valuation ring on $K$ (which is trivial on $k$) considered as an element of $C_K$. This intersection makes sense since each $R$ is a subring of $K$ which contains $k$. \\
Now we consider the semiring case. Let $K:=\mathbb{Q}_{max}(T)$ and $k:=\mathbb{Q}_{max}$, where $\mathbb{Q}_{max}(T)$ is the semifield which we defined in \S \ref{valuationofQmax}. As we mentioned earlier, since any element of $\mathbb{Q}_{max}[T]$ has a tropical solution in $\mathbb{Q}_{max}$, we may consider $\mathbb{Q}_{max}$ as a tropical replacement of an algebraically closed field $k$. 

\subsection{The abstract curve associated to $\mathbb{Q}_{max}(T)$ via strict valuations}\label{mainthm1}
\vspace{0.2cm}

From Proposition \ref{semiringhomvalfunctionfield}, the direct analogue of the set $C_K$ with $K=\mathbb{Q}_{max}(T)$ and $k=\mathbb{Q}_{max}$ is the set $Val(\mathbb{Q}_{max}(T)):=\{\nu_+,\nu_-\}$, where $\nu_+$ is the equivalence class of strict valuations $\nu$ such that $\nu(\overline{T})>0$ and $\nu_-$ is the equivalent class of strict valuations $\nu$ such that $\nu(\overline{T})<0$. This is because the trivial valuation does not have a value group isomorphic to $\mathbb{Z}$ and hence is not discrete.

\begin{mydef}
Let $Val(\mathbb{Q}_{max}(T))=\{\nu_+,\nu_-\}$ be the set of equivalence classes of discrete strict valuations on $\mathbb{Q}_{max}(T)$ which are trivial on $\mathbb{Q}_{max}$.
We define the \emph{abstract curve associated to $K=\mathbb{Q}_{max}(T)$ over $k=\mathbb{Q}_{max}$} to be a topological space which has the underlying set $Val(\mathbb{Q}_{max}(T))$ with the discrete topology.
\end{mydef}

For the notational convenience, let $X=Val(\mathbb{Q}_{max}(T))$ (considered as a topological space).
\begin{mythm}
The abstract curve $X$ is homeomorphic to the subspace of the projective line $\mathbb{P}^1_{\mathbb{F}_1}$ which consists of all closed points. 
\end{mythm}
\begin{proof}
This directly follows from the previous results. We use the same notation as in \S \ref{char}. Consider the following map (of topological spaces):
\begin{equation}\label{homeo}
\varphi: X \longrightarrow \mathbb{P}^1_{\mathbb{F}_1}, \quad \nu_+ \mapsto q_+,\quad \nu_- \mapsto r_-.
\end{equation}
Then clearly $\varphi$ is a homeomorphism from $X$ to the subspace of $\mathbb{P}^1_{\mathbb{F}_1}$ which consists of all closed points. We note that if we do not impose the condition of valuations being discrete then $\varphi$ is no longer a homeomorphism. Since we have to add the trivial valuation as a closed point, however $\mathbb{P}^1_{\mathbb{F}_1}$ only has two closed points.  
\end{proof}

Next, we consider the structure sheaf on $X$. The topological space $X$ has three nonempty open subsets, namely $U=\{\nu_+\}$, $V=\{\nu_-\}$, and $X$ itself. Let $R_+$ be the valuation ring corresponding to $\nu_+$ and $R_-$ be the valuation ring corresponding to $\nu_-$. This makes sense since any two equivalent (classical, strict, hyperfield) valuations on an idempotent semifield have the same valuation ring from Corollary \ref{equihassamerings}. If we directly generalize the classical definition \eqref{sheaf}, we obtain the following:
\[
\mathcal{O}_X(U)=R_+, \quad \mathcal{O}_X(V)=R_-, \quad \mathcal{O}_X(X)=R_+ \cap R_-.
\]
More precisely, we have
\[
R_+=\{\frac{f}{g} \in \mathbb{Q}_{max}(T) \mid \deg (f) - \deg(g) \leq 0\}\cup \{-\infty\},
\]
\[
R_-=\{\frac{f}{g} \in \mathbb{Q}_{max}(T) \mid r_g -r_f \leq 0\}\cup \{-\infty\},
\]
where $r_f$ is defined in $\S \ref{valuationofQmax}$. We note that $R_+$ (resp. $R_-$) has a unique maximal ideal $\mathfrak{m}$ (resp. $\mathfrak{n}$) and if we localize $R_+$ (resp. $R_-$) at $\mathfrak{m}$ (resp. $\mathfrak{n}$) then we obtain $\mathbb{Q}_{max}(T)$. \\
In the spirit of the classical result, one may hope that the set of valuations $X=Val(\mathbb{Q}_{max}(T))$ gives some geometric information about the projective line $\mathbb{P}^1_{\mathbb{Q}_{max}}$ over $\mathbb{Q}_{max}$. However, one can observe that $\mathbb{A}^1:=\Spec (\overline{\mathbb{Q}_{max}[T]})$ already contains a lot more points that $X$. For example, in \cite{noah}, the authors proved that there is an one-to-one correspondence between principle prime ideals of $\overline{\mathbb{Q}_{max}[T]}$ and points of $\mathbb{Q}_{max}$. Hence, the points of the projective line $\mathbb{P}^1_{\mathbb{Q}_{max}}$ over $\mathbb{Q}_{max}$ are a lot more than the points of $Val(\mathbb{Q}_{max}(T))$. This implies that $Val(\mathbb{Q}_{max}(T))$ can not be understood as an abstract curve analogue of $\mathbb{P}^1_{\mathbb{Q}_{max}}$ with strict valuations. In the next subsection, however, we will show that the abstract curve associated to $\mathbb{Q}_{max}(T)$ which is constructed by using hyperfield valuations precisely captures the closed point of the tropical projective line. 
\begin{rmk}\label{absractcurve PROJ}
As we mentioned, the number of closed points of $\mathbb{P}^1_{\mathbb{F}_1}$ is exactly same as the number of points of $Val(\mathbb{Q}_{max}(T))=\{\nu_+,\nu_-\}$. In fact, one may consider that $\nu_0$, which is an equivalence class of the trivial valuation, corresponds to $p_0$ which is the generic point of $\mathbb{P}^1_{\mathbb{F}_1}$. This correspondence can be justified since the classical theorem of abstract nonsingular curves only concerns closed points of a projective nonsingular curve. Hence, at least topologically, one can consider $Val(\mathbb{Q}_{max}(T))$ as an abstract curve analogue of the projective line $\mathbb{P}^1_{\mathbb{F}_1}$.
\end{rmk}
\begin{rmk}
In our construction, $X=Val(\mathbb{Q}_{max}(T))$ is endowed with the discrete topology. This is defined in an ad hoc manner and hence there may be a better definition. Furthermore, contrary to the classical case, the set of global sections $\mathcal{O}_X(X)$ is different from $\mathbb{Q}_{max}$ in this case. On the other hand, we will see that the abstract curve, construct from hyperfield valuations, satisfies this property. 
\end{rmk}

\begin{rmk}
A monoid scheme, first constructed by Deitmar, is essentially toric (see, \cite[Theorem 4.1]{deitmar2008f1}). One of our initial motivations in developing a valuation theory of semirings is to make an analogue of abstract nonsingular curves as we mentioned. It arises from a heuristic idea that an idempotent semiring $\mathbb{R}_{max}[M]$ which is obtained from a monoid $M$ (as a scalar extension, see \cite{noah}) largely depends on $M$, especially when we consider homomorphisms. In particular, we hope that such a theory could shed some light on constructing a positive genus example in $\mathbb{F}_1$-geometry. 
\end{rmk}



% %elsarticle-num
\subsection{The abstract curve associated to $\mathbb{Q}_{max}(T)$ via hyperfield valuations}\label{hypervaluationcurve}
\vspace{0.2cm}

As in the strict valuation case, we see from the proof of Proposition \ref{valR_valfunctionfield} that any non-trivial hyperfield valuation $\nu$ on $\mathbb{Q}_{max}(T)$ has the value group isomorphic to $\mathbb{Z}$. Let $X$ be the set of non-trivial hyperfield valuations on $\mathbb{Q}_{max}(T)$ which are trivial on $\mathbb{Q}_{max}$. Then, from Proposition \ref{valR_valfunctionfield}, we have the following set bijection:
\begin{equation}\label{bijection}
X \simeq \mathbb{Q}_{max} \cup \{\infty\}. 
\end{equation}
There are two choices of topology for $X$. First, we can use the cofinite topology as in the classical case; this should be considered as `Zariski topology' since any tropical polynomial with one variable has finitely many tropical solutions. Secondly, by using the set bijection \eqref{bijection}, we may impose Euclidean topology to $X$. In the current paper, we restrict ourselves to the cofinite topology which is coarser than Euclidean topology. We will consider the other case in \cite{jun2017tropicalabstractcurve} to extend the current work. \\

Now, we construct a structure sheaf $\mathcal{O}_X$ for $X$, where $X$ is considered as a topological space with the cofinite topology. It follows from Corollary \ref{equihassamerings} that any two equivalent hyperfield valuations give the same valuation ring and hence there exists a set bijection between $X$ and the set of valuation rings. Hence, for any open subset $U$ of $X$, we define the following:
\[
\mathcal{O}_X(U):=\bigcap_{R \in U} R, 
\]
where we consider $X$ as the set of valuation rings. One can easily see that $\mathcal{O}_X(U)$ is a presheaf. 

\begin{pro}
$\mathcal{O}_X$ is a sheaf on $X$. 
\end{pro}
\begin{proof}
Note that any inclusion of open sets $U \hookrightarrow V$ induces an injection $\mathcal{O}_X(V) \hookrightarrow \mathcal{O}_X(U)$ in this case. Hence, if $s, t \in \mathcal{O}_X(U)$, then we have $s|_{V}=t|_{V}$. Next, suppose that $U=\bigcup_i U_i$. Since $X$ is equipped with the cofinite topology, $U=X-\{b_1,b_2,...,b_t\}$ for some $t \in \mathbb{N}$. Then from our definition, 
\[
\mathcal{O}_X(U)=\{\frac{f}{g} \in \mathbb{Q}_{max}(T) \mid \nu_{b_i}(f) \leq \nu_{b_i}(g)\quad  \forall i=1,...,t\}
\]
One can easily see that in this case, we can glue given sections $s_i \in \mathcal{O}_X(U_i)$ such that $s_i|_{U_i\cap U_j}=s_j|_{U_i\cap U_j}$ $\forall i,j$.  
\end{proof}

Next, we prove that the only global sections of $X$ are constants as in the classical case. Furthermore, the semiring of sections $\mathcal{O}_X(X-\{\infty\})$ is isomorphic to $\overline{\mathbb{Q}_{max}[T]}$. To this end, we describe valuation rings for each case. We have the following three cases. 

\begin{enumerate}
\item
Let $\nu_{>0}$ be a non-trivial hyperfield valuation on $\mathbb{Q}_{max}(T)$, which is trivial on $\mathbb{Q}_{max}$, such that $\nu_{>0}(\overline{T})>0$. Let $R_{>0}$ be the valuation ring of $\nu_{>0}$. Then $\frac{f}{g} \in \mathbb{Q}_{max}(T)$ is an element of $R_{>0}$ if and only if $\deg(f) \leq \deg(g)$.
\item
Let $\nu_{<0}$ be a non-trivial hyperfield valuation on $\mathbb{Q}_{max}(T)$, which is trivial on $\mathbb{Q}_{max}$, such that $\nu_{<0}(\overline{T})<0$. Let $R_{<0}$ be the valuation ring of $\nu_{<0}$. Then $\frac{f}{g} \in \mathbb{Q}_{max}(T)$ is an element of $R_{>0}$ if and only if $r_f\geq r_g$.
\item
For $b \in \mathbb{Q}$, let $\nu_b$ be a non-trivial hyperfield valuation on $\mathbb{Q}_{max}(T)$, which is trivial on $\mathbb{Q}_{max}$, such that $\nu_{b}(\overline{T})=0$ and $\mathfrak{p}_{\nu_b}=<0\oplus b\odot T>$. Let $R_{b}$ be the valuation ring of $\nu_{b}$. Then $\frac{f}{g} \in \mathbb{Q}_{max}(T)$ is an element of $R_{b}$ if and only if $r_b(f)\geq r_b(g)$.
\end{enumerate}

Then we have the following analogous result. 

\begin{pro}\label{globalsectionpro}
$\mathcal{O}_X(X) \simeq \mathbb{Q}_{max}$. 
\end{pro}
\begin{proof}
Clearly, $\mathbb{Q}_{max} \subseteq \mathcal{O}_X(X)$. Suppose that $x=\frac{f}{g} \in \mathcal{O}_X(X)$. We can write 
\[
x=\frac{f}{g}=\frac{T^n\odot l_1^{a_1}\odot \cdots \odot l_t^{a_t} }{T^e\odot q_1^{b_1}\odot \cdots \odot q_m^{b_m}}, \quad n,e,a_i,b_i \geq 0, 
\]
where $l_i$ and $q_i$ are linear polynomials. Since $x \in \mathcal{O}_X(X) \subseteq  R_{<0}$, one can see that $n \geq e$ and hence we may write
\[
x=\frac{T^s\odot l_1^{a_1}\odot \cdots \odot l_t^{a_t} }{ q_1^{b_1}\odot \cdots \odot q_m^{b_m}}, \quad s,a_i,b_i \geq 0, 
\]
Similarly, when we consider $R_b$, one should have $t=m$ and $\{l_1,...,l_t\}=\{q_1,...,q_m\}$ and if $l_i=q_j$ then $a_i \geq b_j$. It follows that
\[
x=\frac{T^s\odot l_1^{c_1}\odot \cdots \odot l_t^{c_t} }{ 1_{\mathbb{Q}_{max}}}, \quad c_i \geq 0. 
\]
Finally, since $x \in R_{>0}$, we should have that $s+c_1+\cdots +c_t \leq 0$ and hence $s=c_i=0$ for all $i$ since $s$ and $c_i$ are nonnegative. This implies that $x \in \mathbb{Q}_{max}$. 
\end{proof}

In particular, Proposition \ref{globalsectionpro} implies the following. 

\begin{cor}
Let $U:=X - \{\infty\} (=\mathbb{Q} \cup \{-\infty\})$. Then, \[\mathcal{O}_X(U) \simeq \overline{\mathbb{Q}_{max}[T]}.
\]
\end{cor}

\begin{rmk}
Let $U=X-\{x_1,...,x_t\}$ such that $x_i$ are points different from $\infty$ and $-\infty$. Then one may consider $\frac{f}{g} \in \mathcal{O}_X(U)$ as a piecewise linear function which can be defined for all $\mathbb{Q}$ but $\{x_1,...,x_t\}$. 
\end{rmk}

\subsection{Tropical curves and Analytic spaces}\label{tropicalcurve}
\subsubsection{Valuations and tropical curves}
A strict valuation is a homomorphism of semirings which was first introduced in \cite{izhakian2011supertropical} (or see \cite{izavalue} for a brief version). Together with tropical scheme theory introduced by J.Giansiracusa and N.Giansiracusa in \cite{noah}, this leads to some interesting observation. Let us first consider the following example. 

\begin{myeg}\label{easyexample}
Let $S=\mathbb{Q}_{max}[x^{\pm},y^{\pm}]$ be the semiring of Laurent polynomials in two variables with coefficients in $\mathbb{Q}_{max}$. Let $f(x,y)=x+y+3$ and $\tilde{S}:=S/\mathcal{B}(f)$ be a semiring constructed from $S$ and $f$ as in \cite{noah}. The tropical variety $Trop(f)$ defined by $f$ (in $\mathbb{R}^2$) consists of three rays as follows:
\[
X_1=\{(x,x) \mid 3 \leq x, x \in \mathbb{R}\},\quad X_2=\{(3,y) \mid y \leq 3, y \in \mathbb{R}\},
\] 
\[
X_3=\{(x,3) \mid x \leq 3, x \in \mathbb{R}\}.
\]
It follows from \cite{noah} that $X_1 \cup X_2 \cup X_2$ should be $\Hom(\tilde{S},\mathbb{R}_{max})$ and hence the set $X_{strict}$ of strict valuations on $\tilde{S}$ which extend the injection $i:\mathbb{Q}_{max} \hookrightarrow \mathbb{R}_{max}$ is same as the set $(X_1\cup X_2 \cup X_3)$. One can further observe that any two strict valuations $\nu_1$ and $\nu_2$ are equivalent if and only if both are in the same ray. In other words, each equivalence class of strict valuations on $\tilde{S}$ corresponds to a linear piece of $Trop(f)$. 
\end{myeg} 

In general, we observe the following:
\begin{itemize}
\item
Suppose that $X$ is a curve in a torus $T^n$ over a non-Archimedean field $k$. Let $X_{trop}$ be a scheme-theoretic tropicalization (as in \cite{noah}) of a curve $X$ and $S=\mathcal{O}_{X_{trop}}(X_{trop})$ be the semiring of global sections. Then it follows from the results in \cite{noah} that the (set-theoretic) tropicalization $Trop(X)$ of $X$ is equal to $X_{trop}(\mathbb{R}_{max})=\Hom(S,\mathbb{R}_{max})=X_{strict}$. 
\item
As one can see from the Example \ref{easyexample}, each equivalence class of strict valuations may correspond to a linear piece of a tropical curve. 
\end{itemize} 

\begin{rmk}
A strict valuation was first introduced in \cite{izhakian2011supertropical}. However, since it predates the paper \cite{noah}, we do not know whether the authors of \cite{izhakian2011supertropical} knew their definition of a strict valuation could give another interpretation of tropical varieties.  
\end{rmk}

Let $SVal(S)$ and $HVal(S)$ be the set of strict and hyperfield valuations respectively on an idempotent semiring $S$. It follows from Proposition \ref{generator} that $SVal(S)$ can be naturally considered as a subset of $HVal(S)$. Let us first consider the following example.

\begin{myeg}\label{secondexample}
Let $S$, $f(x,y)$, and $\tilde{S}$ be same as Example \ref{easyexample}. Let $SVal(\tilde{S})$ and $HVal(\tilde{S})$ be the set of strict valuations and hyperfield valuations respectively which extend the injection $i:\mathbb{Q}_{max} \hookrightarrow \mathbb{R}_{max}$. Consider the following map:
\begin{equation}\label{hypermap}
\varphi: HVal(\tilde{S}) \longrightarrow \mathbb{R}^2, \quad \nu \mapsto (\nu(x),\nu(y)).
\end{equation}
We claim that the image $\Img(\varphi)$ of $\varphi$ in $\mathbb{R}^2$ is exactly the tropical line $Trop(f)$ defined by $f(x,y)$. In fact, it follows from Proposition \ref{generator} and Example \ref{easyexample} that $\Img(\varphi)$ contains $Trop(f)$. Thus we show that there is no more. Let $\nu$ be a hyperfield valuation on $\tilde{S}$. We use the notation $\oplus$ for the addition of $\mathbb{T}$. We may consider $\nu$ to be a hyperfield valuation on $S$ such that
\begin{equation}\label{condition}
\nu(x+y+3)=\nu(x+y)=\nu(x+3)=\nu(y+3).
\end{equation}
Let us first consider the case when $\nu(x)=\nu(y)$. If $\nu(x) < 3$ then this implies that 
\[
\nu(x+3) \in \nu(x)\oplus\nu(3) =\nu(x)\oplus 3=3. 
\]
On the other hand, we have
\[
\nu(x+y)\in \nu(x) \oplus \nu(y) = \left[-\infty,\nu(x)\right].
\]
Therefore, we should have $3 \in \left[-\infty,\nu(x)\right]$ from the condition \eqref{condition}. However, this only happens when $ 3\leq \nu(x)$. This gives a contradiction. Hence we should have $3 \leq \nu(x)=\nu(y)$.\\
The second case is when $\nu(x) > \nu(y)$. In this case we have 
\[
\nu(x+y)=\nu(x)\oplus \nu(y)=\nu(x).
\]
We should also have $\nu(x+3) \in \nu(x) \oplus \nu(3)$ and hence $\nu(x) \in \nu(x) \oplus 3$. This is possible only when $3 \leq \nu(x)$. Furthermore, since $\nu(x)=\nu(y+3) \in \nu(y) \oplus 3$ and $\nu(y) < \nu(x)$, we should have $\nu(x)=3$. Therefore, in this case, we obtain $\nu(y) < \nu(x)=3$.\\
The last case when $\nu(x) < \nu(y)$ is similar to the second case and hence we have $\nu(x) < \nu(y)=3$. By combining the three cases, we obtain $Trop(f)$ as the image of $\varphi$. 
\end{myeg}

We remark the following:
\begin{itemize}
\item
In \cite{Henry}, S.Henry introduced a notion of an additive map from monoids to hypergroups to realize commutative monoids as a part of hypergroups. A hyperfield valuation may be considered as a generalization of an additive map from monoids to semirings. %from monoids to hypergroups which was first introduced by S.Henry in \cite{Henry} to realize commutative monoids as a part of hypergroups.
\item
In the view of Example \ref{secondexample}, one may expect that J.Giansiracusa and N.Giansiracusa's tropical scheme theory can be generalized by including the hyperfield $\mathbb{T}$ together with either additive maps of \cite{Henry} or hyperfield valuations. We, however, do not pursue this idea in this paper. 
\end{itemize}

\subsubsection{Hyperfield valuations and analytic spaces}

Our definition of hyperfield valuations provides new perspective on Berkovich's theory of analytic spaces through the theory of hyperring schemes which has been investigated by the author in \cite{jun2015algebraic}.  \\
Let $A$ be a commutative ring. In this case, a hyperfield valuation $\nu:A\to \mathbb{T}$ is just a homomorphism of hyperrings (by considering $A$ and $\mathbb{T}$ in the category of hyperrings). Note that this is not the case for semirings since semirings and hyperrings are not in the same category. It follows that the set of hyperfield valuations on $A$ is the set $\Hom(A,\mathbb{T})$ of homomorphisms of hyperrings.\\
In \cite{baker2016matroids}, Baker and Bowler made the following observation:
\begin{obs}(\cite[Example 4.3]{baker2016matroids})\label{obs}
Let $A$ be a commutative ring. Then a homomorphism (of hyperrings) $\nu:A \to \mathbb{T}$ is the same thing as a prime ideal $\mathfrak{p} \in \Spec A$ together with a real valuation on $\Frac(A/\mathfrak{p})$. In other words, there is a canonical bijection between the set $\Hom(A,\mathbb{T})$ of homomorphisms and the Berkovich spectrum $\mathcal{M}(A)$ of $A$. It follows that the functor $\mathcal{M}$ is representable by the tropical hyperfield $\mathbb{T}$.    
\end{obs}

By implementing the notion of hyperring schemes in \cite{jun2015algebraic}, the above observation can be rephrased in the following way: the Berkovich spectrum $\mathcal{M}(A)$ of $A$ can be identified with the set $\Hom(\Spec \mathbb{T},\Spec A)=\Hom(A,\mathbb{T})$ of `$\mathbb{T}$-rational points' of an affine scheme $\Spec A$, where $\Hom(\Spec \mathbb{T},\Spec A)$ is the set of morphisms of locally hyperringed spaces. In fact, one can globalize Observation \ref{obs}. Furthermore, since any totally ordered abelian group $\Gamma$ can be enriched to a hyperfield (see, \cite{viro}), we can consider a higher rank valuation in a similar manner. We will pursue this idea in \cite{jun2016berkovich} in connection to tropical geometry. 

\begin{rmk}
In \cite{con3}, Connes and Consani proved that the functor $\Spec$ is representable by the Krasner hyperfield $\mathbf{K}$. Precisely, they proved the following theorem: Let $X$ be a scheme over $\mathbb{Z}$. Then there exists a bijection (of sets):
\begin{equation}\label{representable}
X=\Hom(\Spec\mathbf{K},X),
\end{equation}
where $\Hom(\Spec\mathbf{K},X)$ is the set of morphisms of locally hyperringed spaces. This theorem can be considered as the trivial valuation case of the $\mathcal{M}(A)$.
\end{rmk}

\bibliography{Val_Semirings}\bibliographystyle{alpha}
\end{document}